\documentclass[11pt]{article}
\usepackage{latexsym,amsmath,color,amsthm,amssymb,epsfig,graphicx,mathrsfs}
\usepackage{graphicx}
\usepackage{amssymb}
\usepackage[left=1in,top=1in,right=1in,bottom=1in]{geometry}
\usepackage[linktocpage=true]{hyperref}
\usepackage{setspace}
\usepackage{amssymb, amsmath, amsthm, graphicx,mathrsfs}
\usepackage{caption}

\usepackage{comment}
\def\qed{\hfill\ifhmode\unskip\nobreak\fi\quad\ifmmode\Box\else\hfill$\Box$\fi}
\def\ite#1{\hfill\break${}$\hbox to 50pt {\quad(#1)\hfill}}
\def\cA{{\mathcal A}}
\def\cB{{\mathcal B}}
\def\cC{{\mathcal C}}

\def\cH{{\mathcal H}}

\newtheorem{thm}{Theorem}[section]

\newtheorem{const}[thm]{Construction}
\newtheorem{definition}[thm]{Definition}

\newtheorem{lemma}[thm]{Lemma}

\newtheorem{claim}[thm]{Claim}
\newtheorem{fact}[thm]{Fact}

\begin{document}

\title{\vspace{-0.5in} On 2-connected hypergraphs with  no long cycles 
 }

\author{
{{Zolt\'an F\"uredi}}\thanks{
\footnotesize {Alfr\' ed R\' enyi Institute of Mathematics, Hungary.
E-mail:  \texttt{z-furedi@illinois.edu}. 
Research supported in part by the Hungarian National Research, Development and Innovation Office NKFIH grant KH-130371, and  by the Simons Foundation Collaboration Grant 317487.
}}
\and{{Alexandr Kostochka}}\thanks{
\footnotesize {University of Illinois at Urbana--Champaign, Urbana, IL 61801
 and Sobolev Institute of Mathematics, Novosibirsk 630090, Russia. E-mail: \texttt {kostochk@math.uiuc.edu}.
 Research 
is supported in part by NSF grant  DMS-1600592
and grants 18-01-00353A and 19-01-00682  of the Russian Foundation for Basic Research.
}}
\and{{Ruth Luo}}\thanks{University of Illinois at Urbana--Champaign, Urbana, IL 61801, USA. E-mail: {\tt ruthluo2@illinois.edu}. Research of this author
is supported in part by Award RB17164 of the Research Board of the University of Illinois at Urbana-Champaign.
}}

\date{ \today}
\maketitle

\vspace{-0.3in}

\begin{abstract}
 We give an upper bound for the maximum number of edges in an $n$-vertex 2-connected $r$-uniform hypergraph with  no Berge cycle of length $k$ or greater, where $n\geq k \geq 4r\geq 12$. For $n$ large with respect to $r$ and $k$, this bound is  sharp and is significantly stronger than the bound without restrictions on connectivity. It turned out that it is simpler to prove the bound for the broader class of Sperner families where the size of each set is at most $r$. For such families, our bound is sharp for all  $n\geq k\geq r\geq 3$.
 
\end{abstract}

\section{Introduction}

\subsection{Basic definitions}

The {\em upper rank} of a hypergraph $\cH$ is the size of a largest edge. For brevity, instead of saying 
``a hypergraph of upper rank $r$" we will say ``an $r^-$-{\em graph}". When every edge has size $r$, i.e., $\cH$ is $r$-uniform, we call $\cH$ an ``$r$-graph".

A hypergraph $\cH$ is {\em Sperner} if no edge of $\cH$ is contained in another edge. In particular,
a Sperner hypergraph has no multiple edges, and all simple uniform hypergraphs are Sperner.

%
%

\begin{definition} A {\bf Berge cycle} of length $\ell$ in a hypergraph is a set of $\ell$ distinct vertices $\{v_1, \ldots, v_\ell\}$ and $\ell$ distinct edges $\{e_1, \ldots, e_\ell\}$ such that $\{ v_{i}, v_{i+1} \}\subseteq   e_i$ with indices taken modulo $\ell$. The vertices $\{v_1, \ldots, v_\ell\}$ are called {\bf base vertices} of the Berge cycle.

  A {\bf Berge path} of length $\ell$ in a hypergraph  is a set of $\ell+1$ distinct vertices $\{v_1, \ldots, v_{\ell+1}\}$ and $\ell$ distinct hyperedges $\{e_1, \ldots, e_{\ell}\}$ such that $\{ v_{i}, v_{i+1} \}\subseteq   e_i$ for all $1\leq i\leq \ell$. The vertices $\{v_1, \ldots, v_{\ell+1}\}$ are called {\bf base vertices} of the Berge path.
\end{definition}

\begin{definition}
The {\bf incidence bigraph} of a hypergraph $\cH$ is the bipartite graph $I(\cH) = (A, Y; E)$ such that $A = E(\cH)$, $Y = V(\cH)$ and for $a \in A$, $y \in Y$, $ay \in E(I(\cH))$ if and only if $y \in a$ in $\cH$.
\end{definition}

A cycle $C$ of length $2 \ell$ in $I(\cH)$ corresponds to a Berge cycle of length $\ell$ in $\cH$ with the set of base vertices $C \cap Y$ and the set of edges $C \cap A$. Similarly, a path $P$ of length $2\ell+1$ (vertices) in $I(\cH)$ with endpoints in $Y$ corresponds to a Berge path of length $\ell$ in $\cH$ with the set of base vertices $P \cap Y$ and the set of edges $C \cap A$. 

\begin{definition}A hypergraph $\cH$ is called {\bf 2-connected} if its incidence bigraph is $I(\cH)$ is $2$-connected.
\end{definition}

So $\cH$ is 2-connected if it is connected and has no {\em cut
vertex} $v \in V(\cH)$
 (i.e., a partition of $V(\cH)=\{ v\}\cup V_1\cup V_2$, $|V_i|\geq 1$, such
that every edge is contained in either $\{ v\}\cup V_1$ or $\{ v\}\cup
V_2$), nor does it have a {\em cut edge} (i.e., an edge $e \in \cH$ and a
partition of $V(\cH) = V_1 \cup V_2$, $|V_i|\geq 1$, such that every edge
$f\neq e$ is contained in either $V_1$ or in $V_2$).

Let $\cH$ be a hypergraph  and  $p$ be an integer. The {\em $p$-shadow}, $\partial_p \cH$, is the collection of the $p$-sets that lie in some edge of $\cH$. In particular, we will often consider the 2-shadow $\partial_2 \cH$ of an $r$-uniform hypergraph $\cH$.  Each edge $e$ of $\cH$ yields in  $\partial_2 \cH$ a  clique on $|e|$ vertices.

\subsection{Graphs without long cycles}

The classic Tur\' an-type result on graphs without long cycles is:
\begin{thm}[Erd\H{o}s and Gallai~\cite{ErdGal59}]\label{ErdGallaiCyc}
Let $k\geq 3$ and let $G$ be an $n$-vertex graph with more than $\frac{1}{2}(k-1)(n-1)$ edges.
Then $G$ contains a cycle of length at least $k$.
\end{thm}

This bound is sharp for infinitely many $n$: when $k-2$ divides $n-1$, the circumference of each connected $n$-vertex graph whose blocks (maximal connected subgraphs with no cut vertices) are cliques of order $k-1$ is only $k-1$. 

There have been several alternate proofs and sharpenings of the Erd\H{o}s-Gallai theorem
including results by Woodall~\cite{WoodallA}, Lewin~\cite{Lewin}, Faudree and Schelp~\cite{FaudScheB,FaudSche75}, and Kopylov~\cite{Kopy}. See~\cite{FS224} for further details.

The strongest version was that of Kopylov who improved the Erd\H{o}s--Gallai bound for 2-connected graphs. To state the theorem, we first introduce the family of extremal graphs. 

\begin{const}\label{cons1}
Fix $k\geq 4$, $n \geq k$, $\frac{k}{2} > a\geq 1$. Define the $n$-vertex graph $H_{n,k,a}$ as follows.
 The vertex set of $H_{n,k,a}$ is partitioned into three sets $A,B,C$ such that $|A| = a$, $|B| = n - k + a$ and $|C| = k - 2a$
 and the edge set of $H_{n,k,a}$ consists of all edges connecting $A$ with $B$ and all edges in $A \cup C$.
 \end{const}

Note that when $a \geq 2$, $H_{n,k,a}$ is 2-connected, has no cycle of length $k$ or longer, and $e(H_{n,k,a}) = {k-a \choose 2} + (n-k+a)a$.

\begin{figure}[!ht]
  \centering
    \includegraphics[width=0.25\textwidth]{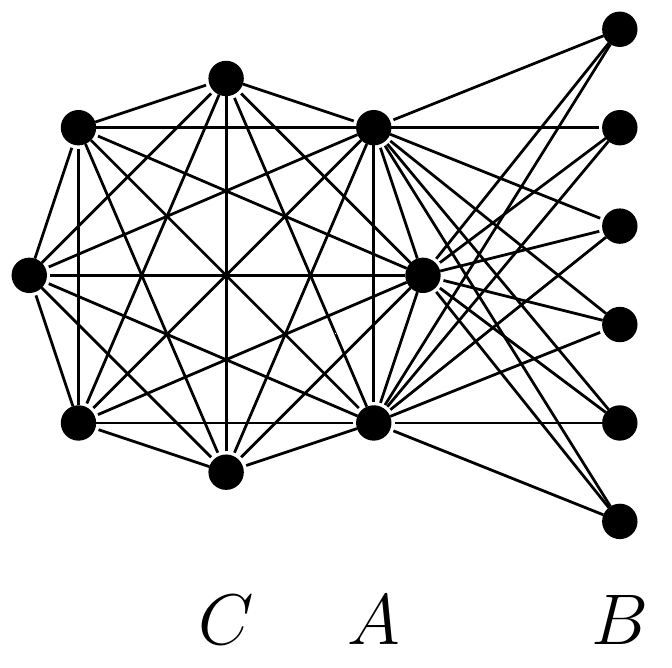}
  \caption{$H_{14,11,3}$}
\end{figure}

\begin{thm}[Kopylov~\cite{Kopy}]\label{kopycycle} Let $n \geq k \geq 5$ and let $t = \lfloor \frac{k-1}{2}\rfloor$. If $G$ is a $2$-connected $n$-vertex graph with \[e(G) > \max\{e(H_{n,k,2}), e(H_{n,k,t})\},\]
then  $G$ has a cycle of length at least $k$.
\end{thm}
Furthermore, Kopylov's proof yields that the only sharpness examples are the graphs $G_{n,k,t}$ and $G_{n,k,2}$. 
See~\cite{FS224} for details.

\subsection{Hypergraphs without long Berge cycles}

%
%
%
%

Recently, several interesting results were obtained for Berge paths and cycles. Notably, the results depend on the relationship between $k$ and $r$.

\begin{thm}[Gy\H{o}ri, Katona, and Lemons~\cite{lemons}]\label{EGp} Let $\mathcal H$ be an $n$-vertex $r$-graph  with no Berge path of length $k$. If $r \geq k \geq 3$, then $e(\mathcal H) \leq \frac{(k-1)n}{r+1}$. If $k > r+1 > 3$, then $e(\mathcal H) \leq \frac{n}{k}{k \choose r}$. 
\end{thm}

Later, the remaining case $k = r+1$ was resolved by Davoodi, Gy\H{o}ri, Methuku, and Tompkins~\cite{davoodi}.

Furthermore, the bounds in Theorem~\ref{EGp} and in~\cite{davoodi} are sharp for each $k$ and $r$ for infinitely many $n$.

Gy\H{o}ri, Methuku, Salia, Tompkins, and Vizer~\cite{connp} proved an asymptotic version of the Erd\H{o}s--Gallai theorem for Berge paths in {\em connected} hypergraphs whenever $r$ is fixed and $n$ and $k$ tend to infinity.  

\begin{thm}[Gy\H{o}ri, Methuku, Salia, Tompkins, and Vizer~\cite{connp}]\label{connpath}
Let $r$ be given. 
Let $\cH_{n,k}$ be a largest $r$-uniform connected $n$-vertex hypergraph with no Berge path of length $k$. Then \[\lim_{k \to \infty} \left(\lim_{n \to \infty} \frac{e(\cH_{n,k})}{k^{r-1}n}\right) = \frac{1}{2^{r-1}(r-1)!}.\]
\end{thm}

For Berge cycles, the exact result for $k\geq r+3$ was obtained in~\cite{FKL}:
\begin{thm}[F\"uredi, Kostochka and Luo~\cite{FKL}]\label{FKL}
Let $k \geq r+3 \geq 6$, and let $\mathcal H$ be  an $n$-vertex $r$-graph   with no Berge cycles of length $k$ or longer. Then $e(\mathcal H) \leq \frac{n-1}{k-2}{k-1 \choose r}$. 
\end{thm}
This theorem is a hypergraph version of Theorem~\ref{ErdGallaiCyc} for $k\geq r+3$. 
The case of $ k \leq r-1$ was resolved by Kostochka and Luo~\cite{KL}.

\begin{thm}[Kostochka and Luo~\cite{KL}]\label{KL}
Let $k \geq 4, r \geq k+1$, and let $\mathcal H$ be an $n$-vertex $r$-graph with no Berge cycles of length $k$ or longer. Then $e(\mathcal H) \leq \frac{(k-1)(n-1)}{r}$. 
\end{thm}

Recently,  Ergemlidze, Gy\H{o}ri, Methuku, Salia, Thompkins, and Zamora~\cite{EGMSTZ} extended the results 
to $k\in \{  r+1,r+2\}$, and Gy\H{o}ri, Lemons, Salia, and Zamora~\cite{k=r} extended the results to $k=r$. 
\begin{thm}[Ergemlidze et al.~\cite{EGMSTZ}]\label{EGMSTZ}
If $k \geq 4$
 and  $\mathcal H$ is  an $n$-vertex $r$-graph  with no Berge cycles of length $k$ or longer, then $k =r+1$ and $e(\mathcal H) \leq n-1$, or $k= r+2$ and $e(\mathcal H) \leq \frac{n-1}{k-2}{k-1 \choose r}$. 
\end{thm}

\begin{thm}[Gy\H{o}ri et al.~\cite{k=r}]
If $r \geq 3$
 and  $\mathcal H$ is  an $n$-vertex $r$-graph  with no Berge cycles of length $r$ or longer, then $e(\cH) \leq \max\{\lfloor\frac{n-1}{r}\rfloor (r-1), n-r+1\}$.  
\end{thm}

%

Theorems~\ref{FKL} and~\ref{KL} are sharp for each $k$ and $r$ for infinitely many $n$. Furthermore, the present authors also proved  in~\cite{FKL_2} exact bounds for all $n$ when $k \geq r+4$. 

For $r \geq k+1$, bounds for 2-connected hypergraphs stronger than for the general case were  found  in~\cite{KL}, although
 they are not known to be sharp. 

\begin{thm}[Kostochka and Luo~\cite{KL}]\label{KL2}
Let $k \geq 4, r \geq k + 1$, and let $\cH$ be an $n$-vertex $2$-connected, $r$-uniform hypergraph with no Berge cycle of length $k$ or longer. Then $e(\cH) \leq \max\{k-1, \frac{k}{2r-k+2}(n-1)\}$.
\end{thm}

In this paper, we find sharp bounds on the maximum number of edges in a 2-connected $r$-uniform hypergraph without Berge cycle of length $k$ or longer in the case $k \geq 4r$ for $n>k2^r$. We do this by proving a more general sharp bound for 
 Sperner  $r^-$-graphs. 

\section{Results}

\subsection{$2$-connected hypergraphs without long Berge cycles}

Our goal is to prove a version of Kopylov's theorem for hypergraphs, i.e., to find the maximum number of edges in a 2-connected hypergraph with no Berge cycle of length $k$ or greater.

Define \[f(n,k,r,a) := {k-a \choose \min\{r, \lfloor\frac{k-a}{2}\rfloor\}} + (n-k+a){a \choose \min\{r-1, \lfloor a/2 \rfloor\}}.\]

Also define \[f^*(n,k,r,a) : = {k-a \choose r} + (n-k+a) {a \choose r-1}.\] Note that $f(n,k,r,a) = f^*(n,k,r,a)$ whenever $r \leq \lfloor (k-a)/2 \rfloor$ and $r-1 \leq \lfloor a/2 \rfloor$.  Our main result is:

\begin{thm}\label{main2conn}
Let $n \geq k \geq r\geq 3$. If $\cH$ is an $n$-vertex Sperner $2$-connected $r^-$-hypergraph with no Berge cycle of length $k$ or longer, then $e(\cH) \leq \max\{f(n,k,r,\lfloor (k-1)/2 \rfloor), f(n,k,r,2)\}$. 
\end{thm}

This bound is sharp. To see this, we construct a series of hypergraphs (not necessarily uniform). The following can be viewed as a hypergraph version of Construction~\ref{cons1}.

\begin{const}\label{cons2}
 For $n \geq k \geq r$, $1\leq a \leq \lfloor (k-1)/2 \rfloor$, let $\cH_{n,k,r,a}$ be the hypergraph with vertex set $A \cup B \cup C$ such that $|A| = k-2a$, $|B| = a, |C| = n-(k-a)$. The edge set of $\cH_{n,k,r,a}$ is the family  \[\{e \subseteq A \cup B: |e| = \min \{r, \lfloor (k-a)/2 \rfloor\}\} \cup \{c \cup e': c \in C, e' \subseteq B, |e'| = \min\{r-1, \lfloor a/2 \rfloor\}\}.\]
 \end{const}

For $a \geq 2$, $\cH_{n,k,r,a}$ is 2-connected and contains no Berge cycle of length $k$ or longer. We have that $|E(\cH_{n,k,r,a})| = f(n,k,r,a)$, which is maximized when $a = \lfloor (k-1)/2 \rfloor$ or $a = 2$ by the convexity of $f$ (as a function of $a$, see 
Claim~\ref{convex} in the appendix). 
Furthermore, when $r \leq \lfloor (k-a)/2 \rfloor$ and $r-1 \leq \lfloor a/2 \rfloor$, $\cH_{n,k,r,a}$ is $r$-uniform with $f^*(n,k,r,a)$ edges.

For integers $k \geq r$, let $n_{k,r}$ be the smallest positive integer $n$ such that $f(n,k,r, \lfloor (k-1)/2 \rfloor) \geq f(n,k,r,2)$. Asymptotically $n_{k,r}$ is about $2^{r-1}k/r$. Then as a corollary of Theorem~\ref{main2conn} we obtain the following result for 
$r$-graphs.

\begin{thm}\label{c23}
Let $n \geq n_{k,r} \geq k \geq 4r\geq 12$. If $\cH$ is an $n$-vertex $2$-connected $r$-graph with no Berge cycle of length $k$ or longer, then $e(\cH) \leq f(n,k,r,\lfloor (k-1)/2 \rfloor) = f^*(n,k,r,\lfloor(k-1)/2 \rfloor)$. 
\end{thm}

For $n$ large, this bound is almost $2^{r-1}/r$ stronger than the (exact) bound in Theorem~\ref{FKL} with no restriction on connectivity.
Again we have sharpness example $\cH_{n,k,r,\lfloor (k-1)/2 \rfloor}$.

\subsection{Connected hypergraphs without long Berge path}
We also obtain a  result for connected graphs with no Berge path of length $k$. 

\begin{thm}\label{main_paths} Let $n \geq k \geq r\geq 3$. If $\cH$ is an $n$-vertex Sperner connected $r^-$-graph with no Berge path of length $k$, then $e(\cH) \leq \max\{f(n,k,r,\lfloor (k-1)/2 \rfloor), f(n,k,r,1)\}$.
\end{thm}
For integers $k \geq r$, let $n'_{k,r}$ be the smallest positive integer $n$ such that $f(n,k,r, \lfloor (k-1)/2 \rfloor) \geq f(n,k,r,1)$. Then we obtain the following result for $r$-uniform graphs with no Berge path of length $k$ as a corollary of Corollary~\ref{main_paths}. This improves Theorem~\ref{connpath}.

\begin{thm}\label{c25}
Let $n \geq n'_{k,r} \geq k \geq 4r\geq 12$. If $\cH$ is an $n$-vertex connected $r$-graph with no Berge path of length $k$, then $e(\cH) \leq f(n,k,r,\lfloor (k-1)/2 \rfloor) = f^*(n,k,r, \lfloor (k-1)/2 \rfloor)$. 
\end{thm}

The family $\cH_{n, k, r, \lfloor(k-1)/2 \rfloor}$ again shows sharpness of our bounds.

\section{Proof outline}
The basic idea of the proof is to consider instead of the family of $r$-graphs the larger family of Sperner $r^-$-graphs. Then we can
in some situations shrink some edges keeping the $r^-$-graph Sperner.

We start with a dense Sperner $r^-$-graph $\cH$. By definition, each edge $e$ in $\cH$ yields a  clique of order $|e|$ in the $2$-shadow
of $\cH$. If $\cH$ contains a long Berge cycle $C$, then $\partial_2 \cH$ contains a cycle of the same length.
However, the converse is not always true. So, our first goal is to reduce $\cH$ to a smaller  dense Sperner $r^-$-graph $\cH'$ for which we know that the existence of a long cycle in $\partial_2 \cH'$ implies the existence of a long cycle in $\cH'$ itself.

Our second goal is to  give an upper bound on the maximum size of a Sperner family of cliques of order at most $r$ in the shadow $\partial_2 \cH'$ that does not have long cycles. This automatically yields a bound on $|\cH'|$.

We systematically consider incidence graphs of $r^-$-graphs instead of the $r^-$-graphs themselves, because we find the language of $2$-connected bipartite graphs convenient for our goals.

In Section 4, we prove two results for the maximum number of cliques in graphs without long cycles or paths which will later be applied to the $2$-shadows of $r^-$-graphs. Specifically, we give upper bounds for the size of Sperner families of cliques of size at most $r$ in graphs with bounded circumference and graphs that do not contain long paths between every pair of vertices. 

In Sections 5 and 6, we prove that our hypergraphs have such a dense subhypergraph that we may reduce to, working in the language of incidence bigraphs in Section 5 and the language of hypergraphs in Section 6. In Section 7, we combine the results from Sections 4-6 to prove Theorem~\ref{main2conn}. Finally, in Section 8 we prove Theorem~\ref{main_paths} for Berge paths in connected hypergraphs.

\section{Sperner cliques in graphs}

A set family $H$ is called {\em Sperner} if no element of $H$ is contained in another element of $H$. In particular, every uniform family is Sperner.

The classic proof of LYM Inequality yields also the following result.

\begin{thm}\label{sperner}
Let $H$ be a set of $h$ elements. Let $\cC$ be a Sperner family of subsets of $H$ such that $|C|\leq r$ for each $C \in \cC$. Then $|\cC| \leq {h \choose \min\{r, \lfloor h/2 \rfloor\}}$.
\end{thm}

\subsection{Cliques in graphs with bounded circumference}

In~\cite{luo}, Luo proved an upper bound for the maximum number of cliques in a 2-connected graph with bounded circumference.

\begin{thm}[Luo~\cite{luo}]\label{L}
Let $n, k, r$ be positive integers with $n \geq k$. Let $G$ be an $n$-vertex $2$-connected graph with no cycle of length $k$ or longer. Then the number of copies of $K_r$ in $G$ is at most \[\max\left\{{k-2 \choose r} + (n-k+2){2 \choose r-1}, {\lceil (k+1)/2 \rceil \choose r } + (n-\lceil (k+1)/2 \rceil) {\lfloor(k-1)/2 \rfloor \choose r-1} \right\}.\]

\end{thm}

We will prove a version of Theorem~\ref{L} for Sperner families of cliques.

Recall \[f(n,k,r,a) := {k-a \choose \min\{r,\lfloor\frac{k-a}{2}\rfloor\}} + (n-k+a){a \choose \min\{r-1, \lfloor a/2 \rfloor\}}.\]

For fixed positive integers $n \geq k \geq r$, $f(n,k,r,a)$ is convex over integers $a$ in $[0, \lfloor (k-1)/2 \rfloor]$ (see the appendix for a proof). Thus the value of $f(n,k,r,a)$ is maximized at one of the endpoints of the domain. 

For a graph $G$ and a positive integer $r$, let $N_{\rm{Sp}}(G,r)$ denote the maximum size of a Sperner family $\cC$ of subsets of $V(G)$ such that for each $C \in \cC$, $G[C]$ is a clique of size at most $r$.

\begin{thm}\label{kopyblock}
 Let $n, k, r$ be positive integers with $n \geq k$. Let $G$ be an $n$-vertex $2$-connected graph with no cycle of length $k$ or longer. Then \[N_{\rm{Sp}}(G,r) \leq \max\{f(n,k,r,2), f(n,k,r,\lfloor(k-1)/2\rfloor)\}.\]
\end{thm}

To prove Theorem~\ref{kopyblock}, we use a structural theorem by Kopylov for 2-connected graphs without long cycles.

{\bf Definition}: For a positive integer $\alpha$ and a graph $G$, the \emph{$\alpha$-disintegration} of
a graph $G$ is the  process of iteratively removing from $G$ the vertices with degree
at most $\alpha$  until the resulting graph has minimum degree at least $\alpha + 1$ or is empty.
This resulting subgraph $H(G, \alpha)$ is called the $(\alpha+1)$-{\em core} of $G$. It is well known (and easy)
that $H(G, \alpha)$ is unique and does not depend on the order of vertex deletion. If $H(G, \alpha)$ is the empty graph, then we say that $G$ is $\alpha$-{\em disintegrable}. 

\begin{thm}[Kopylov~\cite{Kopy}]\label{le:Kopy}
 Let $n \geq k \geq 5$ and let $t = \lfloor \frac{k-1}{2}\rfloor$.
Suppose that $G$ is a $2$-connected $n$-vertex graph with no  cycle of length at least $k$.

Then either \\
${}$\quad {\rm (\ref{le:Kopy}.1)} \enskip the  $t$-core $H(G, t)$ is empty, the graph $G$ is $t$-disintegrable; or \\
${}$\quad {\rm (\ref{le:Kopy}.2)} \enskip $|H(G, t)|=s$ for some
$t+2\leq s\leq k-2$, and $H(G,t)= H(G, k-s)$, i.e.,
the rest of the vertices can be removed by a $(k-s)$-disintegration.
\end{thm}

{\em Proof of Theorem~\ref{kopyblock}}. Set $t := \lfloor (k-1)/2 \rfloor$. Let $G$ be an $n$-vertex 2-connected graph with no cycle of length $k$ or longer. Let $\cC$ be a Sperner family  of subsets of $V(G)$ that are cliques of size at most $r$ with $|\cC|=N_{\rm{Sp}}(G, r)$.
Apply Theorem~\ref{le:Kopy} to $G$.  If (\ref{le:Kopy}.1) holds, then every vertex is deleted in the $t$-disintegration. At the time of its deletion, each vertex $v$ has at most $t$ neighbors and by Theorem~\ref{sperner}, is contained in at most ${t \choose \min\{r-1, \lfloor t/2 \rfloor\}}$ cliques of $\cC$ (since each clique containing $v$ has at most $r-1$   other vertices). After $n-k+t$ steps in the disintegration process, the remaining $k-t$ vertices contain at most ${k-t \choose \min\{\lfloor((k-t)/2)\rfloor, r\}}$ elements of $\cC$.  Therefore $|\cC| \leq N_{\rm{Sp}}(G, r) \leq f(n,k,r,t)$. 
 
 Now suppose (\ref{le:Kopy}.2) holds. Then we consecutively delete vertices of degree at most $k-s$ until we arrive at the core $H(G,t)$
  of size $s$. As in the previous case, when deleting a vertex $v$ of degree at most $k-s$, we remove at most ${k - s \choose \min\{(k-s)/2, r-1\}}$ cliques of $\cC$ containing $v$. Since $H(G,t)$ contains at most ${s \choose \min\{s/2, r\}} = {k - (k - s) \choose \min\{(k-(k-s))/2 \rfloor, r\}}$ cliques in $\cC$,  we obtain $$|\cC| = N_{\rm{Sp}}(G,r) \leq f(n,k,k-s) \leq \max\{f(n,k,r,2), f(n,k,r,t)\}.$$ 
  The last inequality holds by the convexity of  $f$.  
\qed

\subsection{$k$-path connected graphs}

A graph $G$ is {\em $\ell$-hamiltonian} if for each linear forest $L$ with $\ell$ edges (and no isolated vertex) on the vertex set $V(G)$  
 there is a hamiltonian cycle in $G \cup L$ that contains $L$.

A graph $G$ is {\em $k$-path connected} if for each pair of vertices $x,y\in V(G)$, $G$ contains an $x,y$-path  with $k$ or more vertices. In particular, every $n$-vertex  $1$-hamiltonian graph is $n$-path connected. 
The following theorem will be helpful for us.

\begin{thm}[Enomoto~\cite{En}]\label{eno}
Let $G$ be a 3-connected graph on $n$ vertices such that for every pair of vertices $u, v$ such that $uv \notin E(G)$, $d(u) + d(v) \geq t$. Then $G$ is $k$-path connected where $k = \min\{n, 2t-1\}$.
\end{thm}

Define the function
$$h_{\rm{Sp}}(n,\ell, r, d)  := {n - d + \ell \choose \min\{r,\lfloor \frac{n-d+\ell}{2}\rfloor\}} + (d-\ell){d \choose \min\{r-1,\lfloor d/2 \rfloor \}}.$$
Note that $h_{\rm{Sp}}(n,\ell, r,d) = f(n, n+\ell, r, d)$. 
So Claim~\ref{convex} implies (in the appendix) that for given positive $n$, $r$, and $\ell\geq 0$, the function $h_{\rm{Sp}}(n,\ell, r, d)$ is convex
 for $\ell\leq d\leq n$. 
\begin{thm}\label{PPS} Let $n, d,r,\ell$ be integers with $0 \leq \ell < d \leq \left \lfloor \frac{n+\ell-1}{2} \right \rfloor$.
If $G$ is an $n$-vertex graph with minimum degree $\delta(G) \geq d$, and $G$ is not $\ell$-hamiltonian, then
     \[N_{\rm{Sp}}(G, r) \leq \max\left\{ h_{\rm{Sp}}(n,\ell, r, d),h_{\rm{Sp}}(n, \ell, r, \lfloor \frac{n+\ell-1}{2}  \rfloor) \right\}.\]
\end{thm}

\begin{proof}
Let $\cC$ be a Sperner family of cliques of size at most $r$ in $G$. Suppose that $N_{\rm{Sp}}(G, K_r) > h_{\rm{Sp}}(n, \ell, r,\lfloor (n+\ell - 1)/2 \rfloor)$. By a generalization of P\'osa's theorem (Lemma 8 in~\cite{FKLhamcon}),  there exists some $\ell < k < \lfloor (n+\ell - 1)/2 \rfloor$ such that $V(G)$ contains a subset $D$ of $k-\ell$ vertices with degree at most $k$ (and so $k \geq \delta(G) \geq d$).

For each vertex $v \in D$, $v$ is contained in at most ${k \choose \min\{k/2, r-1\}}$ cliques of $\cC$, and $G-D$ contains at most ${n - k+\ell \choose \min \{\lfloor(n-k+\ell)/2 \rfloor, r\}}$ cliques of $\cC$. Hence $|\cC| \leq N_{\rm{Sp}}(G,r) \leq h_{\rm{Sp}}(n,  \ell, r, k)\leq h_{\rm{Sp}}(n,\ell, r, d)$.\end{proof}

Our new result is:

\begin{thm}\label{kpaththm}
Let $n\geq 4$.
Let
 $G$ be an $n$-vertex $2$-connected graph. If
\begin{equation}\label{kpath}
N_{\rm{Sp}}(G,r)> \frac{n-2}{k-3}{k-1\choose \min\{r, \lfloor (k-1)/2 \rfloor\}},
\end{equation}
then $G$ is $k$-path connected.
\end{thm}

{\em Proof of Theorem~\ref{kpaththm}}.
We use induction on $n$. If $n \leq k-1$, then by
 Theorem~\ref{sperner},
$$N_{\rm{Sp}}(G,r)\leq {n \choose \min\{r, \lfloor n/2 \rfloor} =
\frac{n-2}{k-3}\left(
\frac{k-3}{n-2} {n\choose  \min\{r,\lfloor n/2\rfloor\}} \right).
 $$

And for $n \leq k-1$, $$\frac{k-3}{n-2} {n\choose  \min\{r,\lfloor n/2\rfloor\}} \leq \frac{k-3}{(k-1)-2}{k-1 \choose \min\{r, \lfloor (k-1)/2 \rfloor\}}={k-1 \choose \min\{r, \lfloor (k-1)/2 \rfloor\}}.$$ 
Hence~\eqref{kpath} does not hold.

If $n=k$, consider any $x,y\in V(G)$ such that there is no hamiltonian $x,y$-path in $G$. If $xy \in E(G)$, then $G$ is not $1$-hamiltonian, then by Theorem~\ref{PPS} with $d=2$ (since $G$ is $2$-connected),
$$N_{\rm{Sp}}(G, r) \leq \max\{h_{\rm{Sp}}(n,1,r,2), h_{\rm{Sp}}(n, 1,r, \lfloor n/2 \rfloor)) = h_{\rm{Sp}}(n,1,r,2)$$
$$={k-1\choose \min\{r, \lfloor(k-1)/2 \rfloor\}}+2<{k-1\choose \min\{r, \lfloor(k-1)/2 \rfloor\}}
\frac{k-2}{k-3}={k-1\choose \min\{r, \lfloor(k-1)/2 \rfloor\}}
\frac{n-2}{k-3},
$$
and~(\ref{kpath}) again does not hold. If $xy \notin E(G)$, then the graph $G' := G \cup xy$ satisfies $N_{\rm{Sp}}(G', r) \geq N_{\rm{Sp}}(G,r)$, and $G'$ is not 1-hamiltonian. So again we obtain $N_{\rm{Sp}}(G,r) \leq N_{\rm{Sp}}(G', r) \leq {k-1\choose \min\{r, \lfloor(k-1)/2 \rfloor\}}
\frac{n-2}{k-3}$.

Thus from now on we may assume $n \geq k+1$.

\begin{claim}
$G$ is $3$-connected.
\end{claim}

{\bf Proof.} Suppose $\{v_1,v_2\}$ is a separating set. Let $C_1$ be the vertex set of a component of $G-\{v_1,v_2\}$ and $C_2=V(G)-C_1$. For $i=1,2$, let $G_i$ be obtained from $G-C_{3-i}$ by adding edge $v_1v_2$ if it is not in $G$. Let $n_i=|V(G_i)|$.
By construction, each of $G_1$ and $G_2$ is $2$-connected.
Also,
\begin{equation}\label{s1n}
n_1+n_2=n+2\quad\mbox{and}\quad N_{\rm{Sp}}(G, r)\leq N_{\rm{Sp}}(G_1, r)+N_{\rm{Sp}}(2, r).
\end{equation}
By~(\ref{s1n}), some of $G_i$ satisfies~(\ref{kpath}). By symmetry, suppose $G_2$ does. If $x,y\in V(G_2)$, then we are done by induction.
Suppose neither of $x$ and $y$ is in $V(G_2)$. Then by induction, $G_2$ has
a $v_1,v_2$-path $P$ with at least $k$ vertices. Also, the $2$-connected graph $G_1$ has two disjoint paths $P_1$ and $P_2$ from 
$\{x,y\}$ to $\{v_1,v_2\}$. Then $P_1\cup P\cup P_2$ forms a long $x,y$-path.

 Finally, suppose $x\in V(G_2)$ and $y\notin V(G_2)$. Again by induction,  $G_2$ has
a $v_1,x$-path $P$ with at least $k$ vertices. Also, the $2$-connected graph $G_1$ has a $v_1,y$-path $P_1$ that avoids $v_2$.
Then $P\cup P_1$ is what we need.\qed

\begin{claim}
$\delta(G)\geq \frac{k+1}{2}$.
\end{claim}

{\bf Proof.} Suppose $v_1\in V(G)$ and $d(v_1)\leq k/2$. Since $G$ is $3$-connected, we can
choose a neighbor $v_2$ of $v_1$ so that $v_2\notin \{x,y\}$.
Let $G'$ be obtained from $G$ by contracting $v_1$ and $v_2$
into a new vertex that we again will call $v_1$. Since $G$ was $3$-connected, $G'$ is $2$-connected.

Let ${\cal S}_G$ be a maximum Sperner family of cliques of size at most $r$ in $G$.
We construct a  family ${\cal S}'$ of cliques of size at most $r$ in $G'$ from ${\cal S}_G$
by \\
(a) deleting from ${\cal S}_G$ all cliques containing $v_1$; and\\
(b) replacing each clique $S\in {\cal S}_G$ with $v_2\in S$ and $v_1\notin S$ with the clique $S-v_2+v_1$.

We claim that ${\cal S}'$ is Sperner. Indeed, suppose $S_1,S_2\in {\cal S}'$ and $S_1\subset S_2$.
Since ${\cal S}_G$ was Sperner, $v_1\in S_2-S_1$. But then $S_2-v_1+v_2\in  {\cal S}_G$ and
 $S_1\subset S_2-v_1+v_2$.

By construction and Theorem~\ref{sperner}, 
$$|{\cal S}_G|-|{\cal S}'| \leq {d(v_1) \choose \min\{r-1, \lfloor d(v_1)/2 \rfloor\}} \leq {\lfloor k/2 \rfloor \choose \min\{r, \lfloor k/4 \rfloor\}}.$$
But
$${\lfloor k/2 \rfloor \choose \min\{r, \lfloor k/4 \rfloor\}}\leq  \frac{1}{k-3}{k-1\choose \min\{r, \lfloor (k-1)/2\rfloor\}},$$
and hence $G'$ satisfies~(\ref{kpath}). So by the minimality of $G$, graph $G'$ has a long
$x,y$-path. But then $G$ also does.\qed

Applying Theorem~\ref{eno} completes the proof of our theorem.
\qed

\section{Constructing happy incidence bigraphs}

\subsection{Language of layered $r^-$-bigraphs}

A {\em layered bigraph} is a bigraph $G=(A,Y;E)$ in which parts $A$ and $Y$ are ordered.

An $r^-$-{\em bigraph} is a layered bigraph $G=(A,Y;E)$ with $d(a)\leq r$ for each $a\in A$.

A  layered bigraph $G=(A,Y;E)$ is {\em Sperner} if the family $\{N(a): a\in A\}$ is Sperner.
 By definition, if $N(a)=\{v,u\}$  in a Sperner bigraph,
then the codegree of the pair $vu$ is $1$.




In particular, the incidence graph $G_\cH$ of an $r^-$-graph $\cH$ is a Sperner $r^-$-bigraph if and only if $\cH$ is Sperner.


A vertex $a\in A$ of a layered bigraph $G=(A,Y;E)$ is {\em happy}, if the the codegree $d(x,y)$ of each pair
$\{x,y\}\subseteq N(a)$ is at least $d(a)-1$ (and {\em unhappy} otherwise). A  layered bigraph $G=(A,Y;E)$
is {\em happy} if every vertex $a\in A$ is happy.

A vertex $y \in Y$ of degree 2 in is {\em special}, if each of the two neighbors is either unhappy or also has degree 2.

Vertices $x, y \in Y$ and $a \in A$ form a {\em special triple} if $x$ and $y$ are special (in particular they have degree 2), $N(a) = \{x,y\}$, and the other neighbors of $x$ and $y$ are unhappy. 

Given a layered bigraph $G=(A,Y;E)$, let {\em the shadow} $\partial(G)$ be the graph $F$ with vertex set $Y$ such that
$xy\in E(F)$ iff there is $a\in A$ with $\{x,y\}\subseteq N(a)$.

For each graph $H$, the {\em circumference}, $c(H)$, is the length of a longest cycle in $H$.

We first prove a simple corollary of Hall's Theorem.

\begin{lemma}[Folklore] Let $G = (A, B; E)$ be a bipartite graph with no isolated vertices such that for each $a \in A$ and every $b \in N(A)$, $d(a) \geq d(b)$. Then $G$ has a matching  covering $A$.
\end{lemma}

\begin{proof}
Suppose  that $G$ has no matching covering  $A$. By  Hall's Theorem, there is $S\subseteq A$ with $|S| > |N(S)|$. 
Choose a minimum such $S$, say $S = \{a_1, \ldots, a_s\}$. By the minimality of $S$, $G$ has a matching  $M$ covering 
 $S':= S - a_s$, say  $M=\{ a_ib_i: 1 \leq i \leq s-1\}$. Since $|N(S)| \leq s-1$,  we have $N(S) = \{b_1, \ldots, b_{s-1}\}$. 
 So,
\[d(a_1) + \ldots + d(a_{s-1}) + d(a_s) = e(S, N(S)) = d_S(b_1) + \ldots + d_S(b_{s-1}) \leq d(a_1) + \ldots d(a_{s-1}),\] a contradiction.
\end{proof}


\begin{lemma}\label{c2}
Let $r\geq 3$. If $G=(A,Y;E)$ is a happy Sperner $r^-$-bigraph and $\partial(G)$ contains a cycle of length $\ell \geq r$, then $G$ contains a cycle of length $2\ell$.
\end{lemma}


{\bf Proof.} Let $C=x_1,\ldots,x_\ell$ be a cycle of length  $\ell\geq r$ in $\partial(G)$. 
Let $F$ be the bipartite graph with parts $Q=E(C)$ and $A$ such that a pair $(x_ix_{i+1}, a)$ is an edge in $F$
if and only if $\{x_ix_{i+1}\}\subseteq N(a)$. If $\ell \geq r+1$, then since each $a \in A$ has degree less than $\ell$, $a$ is adjacent to at most $d(a) - 1$ pairs $x_ix_{i+1}$. On the other hand, for each edge $(x_ix_{i+1}, a)$  in $F$, $d_F(\{x_ix_{i+1}\})\geq d(a)-1$ since $G$ is happy.
So by the previous lemma, $F$ has a matching that covers $E(C)$, say with $x_ix_{i+1}$ matched to $f(x_ix_{i+1}) \in A$. Then we obtain the cycle $x_1, f(x_1x_2), x_2, f(x_2x_3),\ldots, x_{\ell},f(x_\ell x_1),x_1$ of length $2\ell$ in $G$.

Now suppose $\ell = r$. If for every $a \in A$, $N_G(a) \neq \{x_1, \ldots, x_r\}$, then $d_F(a) \leq d(a) - 1$, and we are done as in the previous case. So suppose there exists an $a$ such that $N_G(a) = \{x_1, \ldots, x_r\}$. Then because $G$ is Sperner, each $a' \in A-a$ is adjacent to at most $r-1$ vertices in $\{x_1, \ldots, x_r\}$, and hence $d_F(a') \leq (r-1)-1$. Consider the graph $F - a$. For $a' \in A - a$, $$d_{F-a}(a') = d_{F}(a') \leq \min\{r-2, d(a') - 1\}.$$
 If some vertex $x_ix_{i+1}$ was adjacent to $a$ in $F$, then $d_F(x_ix_{i+1}) \geq d(a) - 1 = r-1$ and so $d_{F-a}(x_ix_{i+1}) \geq r-2$. Otherwise, for each $x_ix_{i+1}$ not adjacent to $a$ in $F$, and each $a' \in N_F(x_i,x_{i+1})$, $d_{F-a}(x_ix_{i+1}) = d_F(x_ix_{i+1}) \geq d(a') -1$, so we are finished as in the first case.
\qed

The same proof also yields the following Lemma for paths of any length.

\begin{lemma}\label{c2path}
Let $G=(A,Y;E)$ be a happy $r^-$-bigraph. If $\partial(G)$ contains a path with $\ell$ vertices, then $G$ contains a path with $2\ell-1$ vertices with endpoints in $Y$.
\end{lemma}

We will often use the following known property of $2$-connected graphs.

\begin{lemma}\label{c0}
Let $G$ be a $2$-connected graph, $xy\in E(G)$ and $S\subset V(G)$ with $|S|\leq |V(G)|-2$.\\
${}$\quad {\rm (1)} $G-xy$ is $2$-connected iff $G-xy$ has a cycle containing $x$ and $y$;\\ 
${}$\quad {\rm (2)}  the graph $G/S$ obtained by gluing the vertices of $S$ into one vertex $s^*$ is  $2$-connected
iff $s^*$ is not a cut vertex of $G/S$.
\end{lemma}

\subsection{Unhappy $r^-$-bigraphs}


\begin{definition}Let $G=(A,Y;E)$ be a Sperner layered $2$-connected $r^-$-bigraph $G = (A,Y;E)$. A {\bf shrinking of $G$} 
is one of the following operations:\\
${}$\quad  {\rm (1)}  deleting an edge of $G$ incident to an unhappy vertex,\\
${}$\quad  {\rm (2)}  deleting a special vertex  $y \in Y$ and all neighbors $b \in N(y)$ with $d(b) = 2$,\\
${}$\quad  {\rm (3)}  deleting a special triple $x,y \in Y$ and $a \in A$, or\\
${}$\quad  {\rm (4)}  gluing together all but one of the neighbors of some unhappy vertex $a \in X$.
\end{definition}

The goal of this subsection is to prove that unhappy Sperner layered 2-connected 
$r^-$-bigraphs not admitting a shrinking have a special structure and high maximum average degree. The main result of the subsection is
the following lemma.
\begin{lemma}\label{c30}
Suppose  $k \geq r\geq 3$ are integers.
Let   $G=(A,Y;E)$ be a Sperner layered $2$-connected $r^-$-bigraph with $c(G)<2k$ 
that is not happy.
Then either $G$  admits a shrinking  such that the resulting graph $G'$ satisfies\\
${}$\quad  {\rm (S1)}  $G'$ is $2$-connected;\\
${}$\quad  {\rm (S2)}  $|E'|\leq |E|$, $|Y'|\leq |Y|$, and $|E'|+|Y'|<|E|+|Y|$; \\
${}$\quad  {\rm (S3)}  $G'$ is Sperner;\\
${}$\quad  {\rm (S4)}  $|A|-|A'|\leq |Y|-|Y'|$; and \\
${}$\quad  {\rm (S5)}  $c(G')<2k$,\\
 or for every unhappy vertex $a \in A$, there exists three vertices $y_1, y_2, y_3 \in N(a)$ and three subgraphs $B_1, B_2, B_3$ of $G$ such that for $i \in \{1,2,3\}$\\
${}$\quad  {\rm (B1)}  $y_i \in V(B_i)$, $a \notin V(B_i)$, and $y_i$ is the only neighbor of $a$ in $B_i$;\\
${}$\quad  {\rm (B2)}  $B_i$ is $2$-connected and Sperner; \\
${}$\quad  {\rm (B3)}  there exists a $x_i \in Y$ such that $\{a,x_i\}$ separates $B_i$ from $G - B_i$;\\
${}$\quad  {\rm (B4)}  $G - (B_i - x_i) - a$ is Sperner and $2$-connected; and\\
${}$\quad  {\rm (B5)}  for $j \in \{1,2,3\} - \{i\}$, $|V(B_i) \cap V(B_j)| \leq 1$ with equality if and only if $x_i = x_j$. 
\end{lemma}

{\bf Proof.} Suppose,  $G=(A,Y;E)$ is a Sperner layered $2$-connected $r^-$-bigraph with $c(G)<2k$ that is not happy.
Then it has an unhappy vertex $a\in A$. Let $N_G(a)=\{y_1,\ldots,y_t\}$. Since $a$ is unhappy, $t\geq 3$.
Assume that there are no $G'$ satisfying the lemma. We derive a series of properties of such $G$.

A vertex $y_i\in N(a)$ is an $a$-{\em menace}, if there is a vertex $m(a,y_i)\in A-a$ such that $N(a)-y_i\subseteq N(m(a,y_i))$.
Since $G$ is Sperner,
\begin{equation}\label{s0}
\mbox{\em  $G-ay_i$ is Sperner if and only if $y_i$ is not an $a$-menace.}
\end{equation}

For brevity, we call pairs of vertices in $Y$ of codegree $1$ {\em thin} and of  codegree at least $2$ ---  {\em thick}.

\begin{claim}
$N(a)$ contains a thin pair.
\end{claim}

{\bf Proof.} Suppose that all pairs of $N(a)$ are thick pairs. For each  $y_i \in N(a)$, the graph $G_i := G - ay_i$ trivially satisfies (S2), (S4), and (S5) in the definition of shrinking. We will show that $G_i$ is also 2-connected, i.e., it satisfies (S1). Let $y_j, y_k \in N(a) - y_i$. Because every pair of $N(a)$ is thick, there exists distinct vertices $b_{ij}, b_{ik} \neq a$ such that $\{y_i, y_j\} \in N(b_{ij})$ and $\{y_i, y_k\} \in N(b_{ik})$. Applying Lemma~\ref{c0} with the cycle $y_ib_{ij}y_jay_kb_{ik}y_i$ certifies that $G_i$ is 2-connected. 

If for some $1\leq i \leq t$, the graph $G_i$ is Sperner, i.e., satisfies (S3), then we are done. Assume not. Because $a$ is the only vertex with a changed neighborhood in $G_i$, for all $i$ there exists a vertex $b_i$ in $G$ such that $\{y_1, \ldots, y_t\} - \{y_i\} \subset N(b_i)$. Furthermore, for $i\neq j$, $b_i \neq b_j$, otherwise some $N(b_i)$ contains $N(a)$, contradicting the fact that $G$ is Sperner.

In particular, each pair in $N(a)$ belongs in the neighborhoods of $a$ and $d(a) - 2$ additional vertices, contradicting that $a$ is unhappy.\qed

\begin{claim}\label{la2}
All distinct thick pairs  in $N(a)$ are disjoint.
\end{claim}

{\bf Proof.} Suppose not. First we show that there exist some thick pairs $\{y_{i^*}, y_{j^*}\}, \{y_{i^*}, y_{k^*}\}$ and a thin pair $\{y_{s^*},y_{t^*}\}$ such that $s^*, t^* \neq i^*$. Let $\{y_i,y_j\}, \{y_i, y_k\}$ and $\{y_s, y_t\}$ be any intersecting thick pairs of $N(a)$ and a thin pair respectively where without loss of generality, $y_s \notin \{y_i,y_j\}$. If $y_t \neq y_i$ then we are done. If not then consider instead the pair $\{y_s,y_j\}$. If it is thin, then we take this pair instead of $\{y_s,y_t\}$. If it is thick, then we let $\{y_i, y_j\}, \{y_s,y_j\}$ be our intersecting thick pairs with $y_j$ playing the role of $y_{i^*}$ and $\{y_s,y_t\} = \{y_s,y_i\}$ be the thin pair. 

Now consider the graph $G - ay_{i^*}$. As in the previous claim, it satisfies (S2), (S4), and (S5) as well as (S1) in the definition of 
shrinking where we define vertices $b_{i^*j^*}, b_{i^*k^*}$ similarly. Since no other vertex contains the pair $\{y_{s^*},y_{t^*}\}$ in its neighborhood, $G - ay_{i^*}$ is Sperner. \qed

\begin{claim}\label{la3}
The codegree of each  pair  in $N(a)$ is at most $2$.
\end{claim}

{\bf Proof.} Suppose there exist distinct vertices $b_1, b_2 \neq a$ both adjacent to $y_1$ and $y_2$. Since $\{y_1,y_2\}$ is a thick pair, $\{y_1,y_3\}$ and $\{y_2, y_3\}$ are thin by the previous claim. Let $P$ be a shortest path in $G-a$ from $y_3$ to $\{y_1,y_2\}$. Note that if $P$ contains $b_1$ or $b_2$, then by the minimality of $|P|$, either $y_1$ or $y_2$ follows directly after. Therefore we may assume by symmetry that $y_1 \in P$ and $b_2 \notin P$. Consider the graph $G - ay_1$. Trivially it satisfies (S2), (S4), and (S5). Because $\{y_2,y_3\}$ is thin, it also satisfies (S3). Finally, the cycle $y_3Py_1b_2y_2ay_3$ certifies that (S1) is satisfied. \qed
\begin{claim}\label{la4}
If a proper subset  $S$ of $N(a)$ is a separating set in $G$, then $S$ contains an $a$-menace.
\end{claim}

{\bf Proof.} If the claim does not hold, choose a smallest separating subset $S=\{y_1,\ldots,y_s\}$ of $N(a)$
not containing $a$-menaces. Since
$S$ is a proper subset of $N(a)$, $s<t$. Let $D_1$ and $D_2$ be  components of $G-S$, where $D_1$ contains $a$.
By the minimality of $S$,
\begin{equation}\label{s1}
\mbox{\em each $y_i\in S$ has a neighbor in $D_2$.}
\end{equation}
Since $G$ is $2$-connected, there are two $v_t, S$-paths $P_1$ and $P_2$ sharing only $v_t$. 
By symmetry we may assume that $P_1$  avoids $a$. Let  $y_1$
be the end of $P_1$ in $S$. By~(\ref{s1}), there is a $y_1,y_2$-path $P_3$ all whose internal vertices are in $D_2$.

Consider $G'=G-ay_1$. Properties (S2), (S4) and (S5) in the claim of the lemma hold for $G'$ by definition. 
Since $y_1$ is not an $a$-menace, by~(\ref{s0}), $G'$ is Sperner, i.e. (S3) holds.
Cycle
$y_2av_tP_1y_1P_3y_2$ together with Lemma~\ref{c0} show that $G'$ is $2$-connected. Thus, $G'$ satisfies the lemma.\qed

\medskip

\begin{claim}\label{la5}
$N(a)$ has no thick pairs.
\end{claim}

{\bf Proof.} Suppose pair $y_1y_2$ is thick. By Claims~\ref{la2} and~\ref{la3}, $d(y_1y_2)=2$ and the common neighbor 
$b\in A-a$ of $y_1$ and $y_2$ has no other neighbors in $N(a)$. Let $N(b)=\{y_1,y_2,z_1,\ldots,z_s\}$.
Since $G$ is Sperner, $s\geq 1$.

By Claim~\ref{la2}, neither of $y_1$ and $y_2$ is an $a$-menace. So, by Claim~\ref{la4}, $G-y_1-y_2$ contains
an $a,b$-path $P_1$. We may assume that $v_t$ is the second and $z_1$ is the second to last vertices of $P_1$.
Since $d(y_1y_2)=2$, by Claim 5, $z_1 \notin N(a)$. So $v_t \neq z_1$.

{\bf Case 1:} $d(y_1)=2$. Then $d(y_1z_1)=1$ and hence $b$ is unhappy.  So, since $d(y_1y_2)=2$, by Claim~\ref{la2},
$d(y_2z_1)=1$. Consider $G'=G-y_1$. As in the proof of Claim~\ref{la4}, (S2), (S4) and (S5)  hold for $G'$ by definition. 
Cycle
$y_2ay_2P_1b y_2$ together with Lemma~\ref{c0} sertify that $G'$ is $2$-connected, i.e., (S1) holds.
Only the neighborhoods of $a$ and $b$ in $A'$ are distinct from those in $A$. So the fact that
$d(y_2z_1)=d(y_2y_t)=1$ shows that $G'$ is Sperner. This proves Case 1.

{\bf Case 2:} $d(y_1)\geq 3$. Let $c\in N(y_1)-a-b$, where if possible we choose $c$ to be adjacent to $z_1$. Since $G$ is $2$-connected, $G-y_1$ has a shortest path $P_2$ from $c$ 
to $V(P_1)\cup \{y_2\}$. Let $x$ be the end of $P_2$ in $V(P_1)\cup \{y_2\}$.

{\bf Case 2.1:} $x\neq b$. Consider $G'=G-ay_1$. As above, (S2), (S4) and (S5)  trivially hold for $G'$. 
Since only the neighborhood of $a$ in $A'$ is distinct from those in $A$ and  
$d(y_2y_t)=1$,  $G'$ is Sperner. We need now only to show that $G'$ is  $2$-connected. If
$x=y_2$, then cycle $cP_2y_2aP_1by_1c$ certifies this. If $x\in V(P_1)-b$, then our certificate is cycle
$cy_1by_2aP_1(a,x)xP_2c$, where $P_1(a,x)$ denotes the subpath of $P_1$ from $a$ to $x$.

{\bf Case 2.2:} $x= b$. Note that because $x \neq z_1$, by the choice of $c$ and the choice of $P_2$, $z_1 \notin N(c)$ for any $c \in N(y_1) - a - b$. In particular, $d(y_1z_1) = 1$, and so $b$ is unhappy. The second to last vertex of $P_2$ is none of $z_1,y_1,y_2$, so we may assume it is $z_2$.
Consider $G'=G-by_1$. Cycle $cP_2by_2ay_1c$ shows that $G'$ is  $2$-connected. As above, (S2), (S4) and (S5)  trivially hold for $G'$. 
Thus if $G'$ is Sperner, then the claim is proved. If  $G'$ is not Sperner, then $y_1$ is a $b$-menace, and there is a vertex
$g\in A-b$ such that $N(g)\supset \{y_2,z_1,z_2\}$. Since $z_1a\notin E$, $g\neq a$. But then instead of the path
$P_2$, we can consider the path $P_2(c,z_2)z_2gz_1$, and will have Case 2.1.\qed


\begin{equation}\label{s4}
\mbox{\em Let $G'$ be obtained from $G$ by gluing all vertices in $N(a)-y_t$ into one vertex
$y^*$.
}
\end{equation}
(S2) holds for $G'$ trivially. When gluing the vertices, we lose edges only if some pair $y_i, y_j \in N(a)$ have a common neighbor. But because $\{y_i, y_j\}$ is thin, they have no common neighbors other than $a$. Hence $|E'| = |E| - (t-2)$ and $|Y'| = |Y| - (t-2)$ so (S4) holds. Property (S5) is less clear but still is true: If $G'$ has a cycle $C$ of length at least
$2k$, then it must go through $y^*$. Furthermore, if $C$ does not go through $a$, then either $C$ is present in $G$ with
$y^*$ replaced by some $y_i$, or it can be extended through $a$ connecting some $y_i$ and $y_j$. If $C$ does through
$a$, then it uses edges $ay_t$ and $ay^*$; we can modify $C$  in $G$ to a cycle of the same length. Thus, (S5) also holds.

 Since all pairs in $N(a)$ are thin, none of $y_i$ is an $a$-menace. So by Claim~\ref{la4} and Lemma~\ref{c0}, $G'$ is 
 $2$-connected. Again, since all pairs in $N(a)$ are thin, $N_{G'}(a)$ is not contained in any other neighborhood.
 Hence, in order the lemma to fail, by symmetry there are $b_1,b_2\in A - a$ such that $N_G(b_2)-y_2\subset N_G(b_1)$ and
 $y_1b_1\in E$. Note that $b_1$ and $b_2$ each contain exactly one vertex in $N(a)$ ($y_1$ and $y_2$ respectively), and there is $x\in N(b_1)\cap N(b_2)$ such that $x \notin N(a)$.

\begin{claim}\label{la6}
$d(b_2)=2$.
\end{claim}

{\bf Proof.} Suppose $N(b_2)\supseteq \{y_2,x_1,x_2\}$. Then by the definition of $b_1$, $N(b_1)\supseteq \{y_1,x_1,x_2\}$.
So by Claim~\ref{la5} applied to $b_1$ and $b_2$, because the pair $\{x_1,x_2\}$ is thick, both $b_1$ and $b_2$ are happy . Since $G$ is $2$-connected, $G-a$ has a shortest path $P$ from $v_t$ to
$Z=\{y_1,y_2,b_1,b_2,x_1,x_2\}$. Let $z$ be the last vertex of $P$. By symmetry, we may assume 
$z\in \{y_2,b_2,x_2\}$. Consider $G'=G-ay_2$. As before, (S2),(S4) and (S5) hold for $G'$. Since all pairs in $N(a)$ are thin,
$G'$ is Sperner. If $z=y_2$, then the cycle $aPy_2b_2x_2b_1y_1a$ shows that $G'$ is $2$-connected.

So suppose $z\in \{b_2,x_2\}$. Since $b_2$ is happy, there is another $b_3$ adjacent to $y_2$ and $x_2$. By definition, it is distinct 
from $b_1$ and $a$.
So if $z=x_2$ and $P$ does not pass through $b_3$, then we have
cycle $aPx_2b_3y_2b_2x_2b_1y_1a$. Similarly, if $z=b_2$ and $P$ does not pass through $b_3$, then we have
cycle $aPb_2y_2b_3x_1b_1y_1a$. Finally, if $P$  passes through $b_3$, then we have
cycle $aP(a,b_3)b_3y_2b_2x_1b_1y_1a$.\qed

\begin{claim}\label{la7}
$d(y_2)\geq 3$.
\end{claim}

{\bf Proof.} Recall $x = N(b_1) \cap N(b_2)$. Assume $N(y_2)=\{a,b_2\}$. By Claim~\ref{la4}, $G-y_1-y_2$ has an $a,x$-path $P$. We can choose a shortest such path.  Let $c$ be the second to last vertex in $P$.

{\bf Case 1:} $c\neq b_1$. Consider $G'=G-b_2-y_2$. As before, (S2),(S4) and (S5) hold for $G'$. Since all pairs in $N(a)$ are thin,
$G'$ is Sperner. The cycle $aPxb_1y_1a$ shows that $G'$ is $2$-connected.

{\bf Case 2:} $c= b_1$. Let $z$ be the previous to $c$ vertex of $P$. Since all pairs in $N(a)$ are thin, $z\neq v_t$.
If $b_1$ is happy, then there exists a vertex $b_3\neq b_1$ with $\{y_1,x\}\subseteq N(b_3)$. Then $b_3$ can play the role of
$b_1$ in the definition of $b_1$ and $b_2$. In this case, we get Case 1 and are done. Thus, $b_1$ is unhappy. Hence
all pairs in $N(b_1)$ are thin.

If $d(x)=2$, consider $G'=G-b_2-y_2-x$. As before, (S2),(S4) and (S5) hold for $G'$. Since all pairs in $N(a)$ and in $N(b_1)$ are thin,
$G'$ is Sperner. The cycle $aPb_1y_1a$ shows that $G'$ is $2$-connected.

So suppose $b_4\in N(x)-b_1-b_2$. Since $G$ is $2$-connected, $G-x$ has an $b_4,a$-path $P_1$. If $P_1$ does not intersect
$\{b_1,y_1\}$, then we have Case 1 with $P=aP_1b_4x$. So, suppose $u$ is the first vertex in $\{b_1,y_1\}$ that is hit by $P_1$.
Note that if $P_1$ meets $P-u$ before $u$, then we can modify it to avoid intersecting with $\{b_1,y_1\}$. Thus we assume below that this is not the case.

If $u=y_1$, consider $G'=G-ay_1$. As before, (S2),(S4) and (S5) hold for $G'$. Since all pairs in $N(a)$  are thin,
$G'$ is Sperner. The cycle $aPb_1y_1P_1(y_1,b_4)xb_2y_2a$ shows that $G'$ is $2$-connected.
Finally, if $u=b_1$, consider $G'=G-b_1x$. As before, (S2),(S4) and (S5) hold for $G'$. Since all pairs  in $N(b_1)$ are thin,
$G'$ is Sperner. The cycle $aPb_1P_1(b_1,b_4)xb_2y_2a$
 shows that $G'$ is $2$-connected.\qed

\begin{claim}\label{la8}
Set $\{x,y_1,y_2\}$ separates $a$ from $b_1$.
\end{claim}

{\bf Proof.} Suppose not. Then $G-\{x,y_1,y_2\}$ has an $a,b_1$-path $P$. Note that $b_2 \notin P$ since $N(b_2) = \{x,y_2\}$. Let the second vertex of $P$
be $v_t$.

If $b_1$ is happy, then there is $b_3\in A-b_1$ with 
$N(b_3)\supseteq \{y_1,x\}$. Consider $G'=G-ay_1$. As before, (S2),(S4) and (S5) hold for $G'$. Since all pairs in $N(a)$  are thin,
$G'$ is Sperner. We need to show that $G'$ is $2$-connected. If $b_3\in P$, then the cycle
$aP(a,b_3)b_3y_1b_1xb_2y_2a$ certifies this. Otherwise,  the cycle
$aPb_1y_1b_3xb_2y_2a$ certifies this.

So,  $b_1$ is unhappy, and all pairs  in $N(b_1)$ are thin. If $d(y_1)=2$, consider $G'=G-y_1$. 
As before, (S2),(S4) and (S5) hold for $G'$. Since all pairs  in $N(b_1)$ and in $N(a)$ are thin,
$G'$ is Sperner. The cycle $aPb_1xb_2y_2a$
 shows that $G'$ is $2$-connected. 
 
 Thus, $d(y_1)\geq 3$. Let $c\in N(y_1)-a-b_1$. Let $P_1$ be a shortest path in $G-y_1$ from $c$ to
 $V(P)\cup \{x,y_2\}$. Let $z$ be the last vertex of $P_1$. If $z\in V(P)-b_1$, consider $G'=G-ay_1$. As before, (S2),(S4) and (S5) hold for $G'$. Since all pairs in $N(a)$  are thin,
$G'$ is Sperner. The cycle
$aP(a,z)zP_1cy_1b_1xb_2y_2a$ certifies that $G'$ is $2$-connected.

If  $z\in \{b_1,x,y_2\}$, consider $G'=G-b_1y_1$. As before, (S2),(S4) and (S5) hold for $G'$. Since all pairs in $N(b_1)$  are thin,
$G'$ is Sperner. Let $P_2$ denote the path $ay_2b_2xb_1$. Then the cycle
$ay_1cP_1zP_2(z,b_1)b_1Pa$ certifies that $G'$ is $2$-connected.\qed

\begin{claim}\label{la9}
Set $\{x,a\}$ separates $y_2$ from $N(a)-y_2$.
\end{claim}

{\bf Proof.} Suppose not. Let $P$ be a shortest $a,x$-path in $G-y_1-y_2$. 
By Claim~\ref{la8}, $P$ does not go through $b_1$.  Let the second vertex of $P$
be $v_t$. Let $P_1$  be a shortest path in $G-a-x$
from $y_2$ to $(N(a)-y_2)\cup V(P)$. 
 Let $z$ be the last vertex of $P_1$. If $b_1\in V(P)$, then we can take $z=y_1$.
Consider $G'=G-ay_2$. As before, (S2),(S4) and (S5) hold for $G'$. Since all pairs in $N(a)$  are thin,
$G'$ is Sperner. If $z\in N(a)-v_t$ then the cycle
$y_2P_1zaPxb_2y_2$ certifies that $G'$ is $2$-connected. Otherwise, the cycle
$y_2P_1zP(z,a)ay_1b_1xb_2y_2$ does it.\qed 

\bigskip
Let $C_2$ be the vertex set of the component of $G-a-x$ containing $y_2$ and let $G_2=G[C_2\cup \{a,x\}]$.
By Claim~\ref{la9}, $C_2\cap N(a)=\{y_2\}$. If $x$ has no neighbors in $C_2-b_2$, then  by Claim~\ref{la8},
$y_2$ would be a cut vertex, a contradiction. Thus, in view of $b_2$, no vertex in $G_2-a$ separates $x$ from $y_2$.
Since no vertex in $G_2-a$ may separate $\{y_2,x\}$ from any other vertex, we conclude
\begin{equation}\label{s42}
\mbox{\em $G_2-a$ is $2$-connected and the unique neighbor of $a$ in $C_2$ is $y_2$.}
\end{equation}

\begin{claim}\label{la11}
Set $\{x,a\}$ separates $y_1$ from $N(a)-y_1$.
\end{claim}

{\bf Proof.} Suppose not. If $d(b_1) = 2$, then by symmetry of $b_1$ and $b_2$ and the previous claim, we are done. So $d(b_1) \geq 3$. Let $x'\in N(b_1)-y_1-x$. 
Let $P$ be a shortest $a,x$-path in $G-y_1-y_2$. 
By  Claim~\ref{la9}, $P$ does not go through $b_2$.  Let the second vertex of $P$
be $v_t$.

 Let $P_1$  be a shortest path in $G-a-x$
from $\{y_1,b_1\}$ to $ V(P)\cup (N(a)-y_1-y_2)$. Let $z_1$ be the first vertex of $P_1$ and $z_2$ --- the last. If $z_1=y_1$,
consider $G'=G-ay_1$. As above,  (S2),(S4) and (S5) hold for $G'$. Since all pairs in $N(a)$  are thin,
$G'$ is Sperner.  If $z_2\in N(a)-v_t$ then the cycle
$y_1P_1z_2ay_2b_2xb_1y_1$ certifies that $G'$ is $2$-connected. Otherwise, the cycle
$y_1P_1z_2P(z_2,a)ay_2b_2xb_1y_1$ does it. 

So suppose $z_1=b_1$. 

{\bf Case 1:} $b_1$ is unhappy. 
If $z_2\in V(P)$,
then we consider $G'=G-xb_1$. As above,  (S2),(S4) and (S5) hold for $G'$. Since $b_1$ is unhappy, all pairs in $N(b_1)$  are thin, and hence
$G'$ is Sperner.  The cycle
$b_1P_1z_2P(z_2,x)xb_2y_2ay_1b_1$ certifies that $G'$ is $2$-connected. So below  we assume $z_2=y_3$ and $t\geq 4$.

If $d(y_1)=2$,
then we consider $G'=G-y_1$. As above,  (S2),(S4) and (S5) hold for $G'$. Since  all pairs in $N(a)$ and $N(b_1)$  are thin, $G'$ is Sperner.  The cycle
$b_1P_1y_3ay_2b_2xb_1$ certifies that $G'$ is $2$-connected.

Thus there is $b_0\in N(y_1)-a-b_1$. If $G-b_1-y_1$ has a path from $b_0$ to $N(a)-y_1$, then we would have the case $z_1=b_1$ above. Hence there is no such path. But then $G-V(P)-N(a)$ has a $b_0,b_1$-path $P_2$. In this case, we 
consider $G'=G-y_1b_1$. As above,  (S2),(S4) and (S5) hold for $G'$. Since  all pairs in $N(b_1)$  are thin, 
$G'$ is Sperner.  The cycle
$y_1b_0P_2b_1xb_2y_2ay_1$ certifies that $G'$ is $2$-connected.

{\bf Case 2:} $b_1$ is happy. Then there is another common neighbor $b'_1$ of $x$ and $y_1$. Again,
consider $G'=G-ay_1$. As above,  (S2),(S4) and (S5) hold for $G'$. Since all pairs in $N(a)$  are thin,
$G'$ is Sperner.  If $b'_1\notin P_1$ and
$z_2\in N(a)-v_t$ then the cycle
$b_1P_1z_2ay_2b_2xb'_1y_1b_1$ certifies that $G'$ is $2$-connected.  If $b'_1\notin P_1$ and
$z_2\in V(P)$ then the cycle
$b_1P_1z_2P(z_2,a)ay_2b_2xb'_1y_1b_1$ does it. If $b'_1\in P_1$, then we switch the roles of $b_1$ and $b'_1$:
consider the path $P'_1=P_1(b'_1,z_2)$.\qed

\begin{claim}\label{la12}
Vertex $a$ has only one neighbor (namely, $y_1$) in the component $C_1$ of $G-x-a$ containing $y_1$ and $b_1$.
\end{claim}

{\bf Proof.} Otherwise, $\{x,a\}$ would not separate $y_1$ from $N(a)-y_1$, a contradiction to   Claim~\ref{la11}.\qed

\bigskip
Similarly to the definition of $G_2$, let $C_1$ be the vertex set of the component of $G-a-x$ containing $y_1$ and let 
$G_1=G[C_1\cup \{a,x\}]$.
By Claim~\ref{la12}, $C_1\cap N(a)=\{y_1\}$.

\begin{claim}\label{la14}
 $G_1-a$ is $2$-connected. 
\end{claim}

{\bf Proof.} 
{\bf Case 1:} $G-a-b_1$ has an $x,y_1$-path $P$. Then $P+b_1$ forms a cycle in $G_1-a$ containing $x$ and $y_1$.
Since $G$ is $2$-connected and $\{y_1,x\}$ is a separating set in $G_1$, this finishes the case.

{\bf Case 2:} $d(b_1)=2$. Then $y_1$ can play the role of $y_2$, and we are done by~(\ref{s42}).

{\bf Case 3:}
Vertex $b_1$ separates $y_1$ from $x$ in $G_1-a$, and $b_1$ has a neighbor $y'\notin\{x,y_1\}$. 
 If $b_1$ were happy, there would be $b'\neq b_1$ adjacent to $x$ and $y_1$ and we would have Case 1. So, $b_1$ is unhappy.
Let $P_1$ be a shortest path from $y'$ to $\{a,x\}$  in $G-b_1$.
 and $z$
be the  last vertex on $P_1$. 

Suppose first that $z=a$. Then by Claim~\ref{la12}, the second to last vertex of $P_1$ is $y_1$.
Consider $G'=G-y_1b_1$. As above,  (S2),(S4) and (S5) hold for $G'$. Since $b_1$ is unhappy,  all pairs in $N(b_1)$  are thin.
Thus
$G'$ is Sperner.  The cycle
$y'P_1ay_2b_2xb_1y'$ certifies that $G'$ is $2$-connected.

Suppose now that $z=x$. Since Case 1 does not hold, $y_1\notin P_1$.
Consider $G'=G-xb_1$. As above,  (S2),(S4) and (S5) hold for $G'$. Since   all pairs in $N(b_1)$  are thin,
$G'$ is Sperner.  The cycle
$y'P_1xb_2y_2ay_1b_1y'$ certifies that $G'$ is $2$-connected.\qed

\begin{claim}\label{la15}
 $G-C_1$ and $G-C_2$ are $2$-connected Sperner $r^-$-graphs.
\end{claim}

{\bf Proof.} Let $P$ be a shortest $y_3,x$-path in $G-a$. By Claim~\ref{la9} and~\ref{la11}, $P$ avoids $C_1\cup C_2$.
For $i=1,2$, the cycle $y_3Pxb_{3-i}y_{3-i}ay_3$ certifies that $G-C_i$ is $2$-connected. Since the degrees of the vertices in  $G-C_1$ and $G-C_2$ are 
dominated by those in $G$,  $G-C_1$ and $G-C_2$ are $r^-$-graphs. Since $a$ is the only vertex in $A\cap V(G-C_i)$ whose degree decreased w.r.t. $G$
and all pairs in $N(a)$ are thin, $G-C_1$ and $G-C_2$ are  Sperner.
\qed

Now set $B_1 = G_1 - a$, $B_2 = G_2 - a$, and $x_1 = x_2 = x$. Note that the choice of $y_t$ in~\eqref{s4} was arbitrary. So we may repeat the proof instead taking $G''$ to be the graph obtained by gluing $N(a) - y_1$ into a single vertex $y^{**}$. If $G''$ satisfies (S1) - (S5), then we  are  done. Otherwise we find some vertices $y_1', y_2' \in N(a) - y_1$ which play the role of $y_1$ and $y_2$. We may assume that $y_1' \notin \{y_1,y_2\}$ and it is coupled with some vertex $x'$ which plays the role of $x$.

Again, repeating the previous proofs for Claims~\ref{la6}-\ref{la15} with $y_1'$ and $y_2'$, we obtain that either $G$ admits a shrinking, or we can define $G_1'$ similarly to play the role of $G_1$ (defined after Claim~\ref{la12}) for $y_1'$. Let $B_3 = G_1' - a$, $y_3 = y_1'$, and $x_3 = x'$. We now show that (B1) - (B5) hold.

(B1) and (B3) are trivial. Since $G$ was Sperner each vertex of $A \cap V(B_i)$ has the sane neighborhood in $B_i$, $B_i$ is also Sperner. Hence together with~\eqref{s42} and Claim~\eqref{la14}, we get (B2). Claim~\ref{la15} proves (B4).  Claims~\ref{la9} and~\ref{la11} imply that $V(B_1) \cap V(B_2) = \{x\}$, and $y_1' (=y_3)$ is contained in a  component of $G - \{a,x\}$ 
not containing  $y_1$ and $y_2$. In particular, $B_3$ is disjoint from $B_1$ and $B_2$ except possibly at $x'$ if $x' = x$. This proves (B5) and thus the Lemma~\ref{c30}. 
\qed
%
%

\subsection{Consequences of Lemma~\ref{c30}}

This technical lemma implies the following more applicable fact.

\begin{lemma}\label{c2happy}
Suppose  $k\geq 5$, $r\geq 3$ are integers with $k\geq r$. Set $t = \lfloor (k-1)/2 \rfloor$.
Let   $G=(A,Y;E)$ be a Sperner layered $2$-connected $r^-$-bigraph with $c(G)<2k$ 
that is not happy.
Then either $G$ admits a shrinking such that the resulting graph satisfies (S1) - (S5), or there exists an unhappy vertex $a^* \in A$ and some block $B^*$ satisfying the hypothesis of Lemma 2.4 such that $B^*$ is happy and $|A \cap B^*| \leq  {t \choose \min\{r-1, \lfloor t/2 \rfloor\}}(|Y \cap B^*| -2)$.
\end{lemma}

\begin{proof}
Suppose $G$ does not admit any shrinking. By Lemma~\ref{c30}, for each unhappy vertex $a$ we obtain some $\{y_i,x_i,B_i\}$ for $i \in \{1,2,3\}$ satisfying (B1) - (B5). 

\begin{claim}\label{xyk}For each unhappy $a$,  at most one $B_i$ has a $(x_i,y_i)$-path of length $k$ or longer. \end{claim}
\begin{proof}
Suppose without loss of generality that for $i \in \{1,2\}$, there exists a $(y_i, x_i)$-path $P_i$ in $B_i$ of length at least $k$. Recall that $y_1,y_2 \in N(a)$. Let $P_3$ be a $(x_1,x_2)$-path internally disjoint from $V(B_1) \cup V(B_2)$ (where $P_3$ may be a singleton). Then $P_1 \cup P_3 \cup P_2 \cup a$ is a cycle of length at least $2k - 1$, i.e., length at least $2k$.
\end{proof}

Among all vertices in $A$ that are not happy, choose $a$ and a corresponding 2-connected graph $B_1$ from Lemma 2.4  so that (a) $B_1$ does not have a $(x_i,y_i)$-path of length $k$ or longer, and (b) subject to (a), $|V(B_1)|$ is minimized.

Suppose first that $B_1$ contains an unhappy vertex $a'$. By Lemma 2.4, there exists $\{x_i',y_i',B_i'\}$ for $i \in \{1,2,3\}$ satisfying (B1)-(B5) with $a'$.

\begin{claim}
At most one $j \in \{1,2,3\}$ satisfies $V(B_j') \not\subseteq V(B_1)$. 
\end{claim}
\begin{proof}
Suppose without loss of generality $V(B_2') \not\subseteq V(B_1)$ and $V(B_3') \not\subseteq V(B_1)$. Then since $\{x_1,a\}$ separates $B_1$ from $G - (B_1-x) - a$, and $B_2'$ and $B_3'$ are 2-connected, $\{x_1, a\} \subseteq V(B_2')$ and $\{x_1, a\} \subseteq V(B_3')$. But this violates (B5).
\end{proof}

Therefore we may assume $V(B_1'), V(B_2') \subseteq V(B_1)$. By Claim~\ref{xyk}, we can also assume that $V(B_1')$ has no $(x_1', y_1')$-path of length $k$ or longer. Furthermore, since $a' \in V(B_1) - V(B_1')$, $|V(B_1')| < |V(B_1)|$.  But this contradicts the choice of $a$ and $B_1$. Thus $B_1$ cannot have any unhappy vertices, i.e., $B_1$ is happy. 
%
%
%

Consider the shadow $\partial(B_1)$ of $B_1$. By Lemma~\ref{c2path}, $\partial(B_1)$ is not $\lceil (k+1)/2 \rceil$-path connected, otherwise $B_1$ would contain an $(x_1, y_1)$-path of length at least $2\lceil (k+1)/2 \rceil - 1 \geq k$, a contradiction.

Let $\alpha = \lceil (k-1)/2 \rceil$, $\beta = \lfloor (k-1)/2 \rfloor$. 

\begin{claim}$\frac{1}{\alpha - 2}{\alpha \choose \min\{r, \lfloor \alpha/2 \rfloor\}} \leq {\beta \choose \min\{r-1, \lfloor \beta/2 \rfloor\}}$.
\end{claim}
\begin{proof}
First suppose $\alpha = \beta$, i.e., $k$ is odd. Then the case $\min\{r, \lfloor \alpha/2 \rfloor\} = \alpha/2$ is trivial. Otherwise $\frac{1}{\alpha-2}{\alpha \choose r} = \frac{1}{\alpha - 2}\frac{\alpha-r+1}{r} {\beta \choose r-1}\leq {\beta \choose r-1}$. 
So assume $\alpha = \beta + 1$. If  $\min\{r, \lfloor \alpha/2 \rfloor\} = r$ (so $\min\{r-1,\lfloor \beta/2 \rfloor\ = r-1$), then we have $\frac{1}{\alpha-2}{\alpha \choose r} = \frac{1}{\beta - 1}\frac{\beta+1}{r} {\beta \choose r-1}\leq {\beta \choose r-1}$. Otherwise if $\lfloor \alpha/2 \rfloor  < r$, then $\lfloor \beta/2 \rfloor \leq r-1$, and $\frac{1}{\alpha-2}{\alpha \choose \lfloor \alpha/2 \rfloor} =\frac{1}{\beta-1}{\beta + 1 \choose \lfloor (\beta+1)/2 \rfloor} = \frac{1}{\beta - 1} \frac{\beta+1}{\lfloor (\beta+1)/2 \rfloor} {\beta \choose \lfloor (\beta+1)/2 \rfloor - 1} \leq {\beta \choose \lfloor \beta/2 \rfloor}.$
\end{proof}

Therefore because $\partial(B_1)$ is not ($\alpha + 1$)-path connected, by Theorem~\ref{kpaththm} and the previous claim, \[|A\cap B_1| \leq N_{\rm{Sp}}(\partial(B_1), r) \leq \frac{|Y \cap B_1| - 2}{\alpha - 2} {\alpha \choose \min\{r, \lfloor \alpha/2 \rfloor\}} \leq (|Y \cap B_1| - 2) {\beta \choose \min\{r-1, \lfloor \beta /2 \rfloor\}}.\]
\end{proof}

\section{Constructing happy $r^-$-graphs}

In this section, we translate Lemma~\ref{c30} into the language of $r^-$-graphs. We also refine it.
\subsection{Unhappy $r^-$-graphs}
A Sperner $r^-$-graph $\cH$ is {\em happy} if its layered incidence bigraph $I(\cH)$ is happy, and is {\em unhappy}
otherwise. The happy and unhappy vertices in $I(\cH)$ correspond to happy and unhappy edges in $\cH$.

For an unhappy edge $e$ in an unhappy $r^-$-graph $\cH$ and a vertex $v\in e$, let $F(\cH,e,v)$ denote the 
$r^-$-graph obtained from $\cH$ by replacing $e$ with $e-v$.

A vertex $v$ of degree $2$ in an unhappy $r^-$-graph $\cH$ is {\em special} if each of the two incident edges,
say   $e_1$ and $e_2$, is either unhappy or
a graph edge (i.e., contains exactly two vertices). If $v$ is special and incident with $e_1$ and $e_2$, then $F(\cH,v,e_1,e_2)$ is the $r^-$-graph obtained from $\cH$ by deleting
$v$ and for $i=1,2$ deleting $e_i$ if $|e_i|=2$ and replacing $e_i$ with $e_i-v$ otherwise.

A graph edge $vu$  in an unhappy $r^-$-graph $\cH$ is {\em special} if both $v$ and $u$ are special, and both adjacent to $vu$ edges are unhappy. If $vu$ is special and adjacent to $e_1$ and $e_2$, then $F(\cH,vu)$ is the $r^-$-graph obtained from $\cH$ by deleting
$v$ and  $u$, replacing $e_1$ with $e_1 - v$, and replacing $e_2$ with $e_2 - u$.


A {\em $2$-block} in a  2-connected  $\cH$ is a 2-connected $\cH'\subset \cH$  such that only two vertices of $\cH'$ have neighbors outside of $\cH'$.
These two vertices will be called {\em outer vertices of} $\cH'$.

A $2$-block $\cH'$ with outer vertices $x$ and $y$ in an unhappy Sperner $r^-$-graph $\cH$ is {\em special} if \\
$\cH'$ is happy and there is exactly one edge, say $a$, in $G-E(\cH')$ containing $y$, and this edge does not contain $x$.

Given a special $2$-block $\cH'$ with outer vertices $x$ and $y$ in an unhappy Sperner $r^-$-graph $\cH$, the
$r^-$-graph $F(\cH,\cH',x,y)$ is obtained from $\cH$ by deleting all vertices of $\cH'-x-y$ together with the edges containing them and
adding edge $\{x,y\}$ if it is not in $\cH$.

Translating from the language of incidence bipartite graphs to hypergraphs, we obtain the following versions of Lemmas~\ref{c2} and~\ref{c2path} about Berge cycles and Berge paths.

\begin{lemma}\label{c2'}
Let $r\geq 3$. Let $\cH$ be a happy $r^-$-graph. If the 2-shadow $\partial_2\cH$ contains a cycle of length $\ell \geq r+1$, then $\cH$ contains a Berge cycle of length $\ell$ on the same base vertices. Furthermore, if $\partial_2 \cH$ contains a path, then $\cH$ contains a Berge path with the same base vertices.
\end{lemma}

For simplicity, for an $r^-$-graph $\cH$, denote $\sum |E(\cH)| := \sum_{e \in E(\cH)} |e|$. For example, if $\cH$ is $r$-uniform, then $\sum|E(\cH)| = r|E(\cH)|$. We also obtain the following as a corollary of Lemma~\ref{c2happy}.

\begin{lemma}\label{c3}
Suppose  $k\geq r\geq 3$  are integers, and set $t = \lfloor (k-1)/2 \rfloor$.
Let   $\cH$ be a Sperner  $2$-connected $r^-$-graph with $c(\cH)<k$ 
that is not happy.
Then we can obtain a Sperner  $2$-connected $r^-$-graph $\cH'$
 such that\\
${}$\quad  {\rm (i)}    $\sum |E(\cH')|\leq \sum |E(\cH)|$, $|V(\cH')|\leq |V(\cH)|$, and $\sum |E(\cH')|+|V(\cH)|<\sum |E(\cH)|+|V(\cH')|$; \\
${}$\quad  {\rm (ii)}   $|E(\cH)|-|E(\cH')|\leq {t \choose \min\{r-1, \lfloor t/2 \rfloor\}} (|V(\cH)|-|V(\cH')|)$; and \\
${}$\quad  {\rm (iii)}  $c(\cH')<k$\\
using one of the following transformations:\\
${}$\quad  {\rm (T1)}  for an unhappy edge $e$ and $v\in e$, replacing $H$ with $F(H,e,v)$;\\
${}$\quad  {\rm (T2)}  for a special vertex $v$ with incident edges $e_1$ and $e_2$, replace $H$ with $F(H,v,e_1,e_2)$;\\
${}$\quad  {\rm (T3)}  for a special edge $vu$, replace $H$ with $F(H,vu)$;\\
${}$\quad  {\rm (T4)}  glue together all but one vertices of an unhappy  edge;\\
${}$\quad  {\rm (T5)}  for a special $2$-block $H'$ with outer vertices say $x,y$, replace $H$ with $F(H,H',x,y)$.\\
Furthermore, if {\rm (T5)} is not applied, then instead of {\rm (ii)}, we obtain $|E(\cH)|-|E(\cH')|\leq  (|V(\cH)|-|V(\cH')|)$.
\end{lemma}

\subsection{A refinement of Lemma~\ref{c3}}

Suppose we start from    a Sperner  $2$-connected unhappy $r^-$-graph $\cH$ with at least $k$ vertices and $c(\cH)<k$.
 Lemma~\ref{c3} provides that we can obtain from $\cH$ a happy Sperner  $2$-connected  $r^-$-graph in several steps using the following rule at each step:
\begin{equation}\label{e1}
	\mbox{\em  if possible, apply {\rm (T1)}; if not then try {\rm (T2)}, then {\rm (T3)} and so on.}
\end{equation}

We may think that we have started from $\cH=\cH_0$ and after Step $i$ obtain $\cH_i$ from $\cH_{i-1}$ using one of (T1)--(T5).

Claims 2.7--2.8 in the proof of  Lemma~\ref{c3} yield that following~(\ref{e1}), at each Step $i$,
\begin{equation}\label{e2}
\parbox{15cm}{\em  if  {\rm (T1)} is not applied on Step $i+1$,  then in each unhappy edge $a$ of $\cH_i$, thick pairs form a matching,}
\end{equation}
and 
\begin{equation}\label{e3}
\parbox{15cm}{\em  if neither  {\rm (T1)} nor {\rm (T2)} is  applied on Step $i+1$,  then all pairs of vertices in each unhappy edge $a$
of $\cH_i$ are thin.}
\end{equation}

\begin{claim}\label{cla1}
If {\rm (T2)} was applied on Step $i$, then {\rm (T1)} cannot be applied on Step $i+1$.
\end{claim}

{\bf Proof.} Suppose $\cH_{i}=F(\cH_{i-1},v,e_1,e_2)$ and $\cH_{i+1}=F(\cH_{i},e_0,w)$. 

\medskip
{\bf Case 1:} Edge $e_0$ is neither $e_1-v$ nor $e_2-v$.
We want to show that in this case, $e_0$ is unhappy in 
$\cH_{i-1}$ and $\cH'=F(\cH_{i-1},e_0,w)$ is a Sperner  $2$-connected $r^-$-graph satisfying (i)--(iii) with $\cH_{i-1}$ in place of $\cH$.
That would contradict Rule~(\ref{e1}).

To prove the first part (that $e_0$ is unhappy in 
$\cH_{i-1}$), recall that $e_0$ is unhappy in $\cH_i$. But the codegree in $\cH_i$ of each  pair in $V(\cH_i)$ is the same as in $\cH_{i-1}$.

To prove the second part, we use the fact that $\cH'$ can be obtained from $\cH_{i+1}$ by adding back vertex $v$ and for $j=1,2$ constructing $e_j$ either by
adding $v$ to $e_j-v\in \cH_{i+1}$ when $|e_j|\geq 3$ or adding edge $e_j$ when  $|e_j|=2$. Since the incidence graph $I(\cH_{i+1})$
is $2$-connected and this operation corresponds to adding a vertex of degree $2$ or an ear to $I(\cH_{i+1})$,
$I(\cH')$ also is 2-connected. Since $\cH_{i+1}$ is Sperner, and   $\cH'$ differs from it only $e_1,e_2$ and $v$, $H'$ is also Sperner:
new edges are not contained in any old edge because of $v$, and no old edge can be contained in $e_j$, since otherwise it would be contained in $e_j-v$ in $\cH_{i+1}$. Properties (i)--(iii) are trivial.

\medskip
{\bf Case 2:}  $e_0=e_1-v$. In this case,  we know that $e_1$ is unhappy in 
$\cH_{i-1}$ and  want to show that  $\cH'=F(\cH_{i-1},e_1,w)$ is a Sperner  $2$-connected $r^-$-graph satisfying (i)--(iii) with $\cH_{i-1}$ in place of $\cH$. Now $\cH'$ can be obtained from $\cH_{i+1}$ by adding back vertex $v$, adding $v$ to $e_0-w$
 and  constructing $e_2$ either by
adding $v$ to $e_2-v\in \cH_{i+1}$ when $|e_2|\geq 3$ or adding edge $e_2$ when  $|e_2|=2$. The rest is as in Case 1.\qed

Practically the same proof yields the following similar claim.

\begin{claim}\label{cla2}
If {\rm (T3)} was applied on Step $i$, then {\rm (T1)} cannot be applied on Step $i+1$.\qed
\end{claim}

The proof of the next claim is somewhat different.

\begin{claim}\label{cla3}
If {\rm (T4)} was applied on Step $i$, then {\rm (T1)} cannot be applied on Step $i+1$.
\end{claim}

{\bf Proof.} Suppose $\cH_{i-1}$ has an unhappy edge $a=\{y_1,\ldots,y_t\}$ such that
$\cH_{i}$ is obtained from $H_{i-1}$ by gluing $\{y_1,\ldots,y_{t-1}\}$ into a new vertex $y^*$,
 and $\cH_{i+1}=F(\cH_{i},e,w)$. By~(\ref{e3}),
\begin{equation}\label{e4}
\parbox{15cm}{\em  all pairs of vertices in each unhappy edge 
of $\cH_{i-1}$ are thin. In particular, the size of each edge in $\cH_i$ apart from the edge $y^*y_t$ is the same as in $\cH_{i-1}$. }
\end{equation}

{\bf Case 1:} $w\neq y^*$. By~(\ref{e4}), in $\cH_{i-1}$, $|e\cap a|\leq 1$. So, since $e$ is unhappy in $\cH_i$, it is also unhappy 
in $\cH_{i-1}$. We want to show that  $\cH'=F(\cH_{i-1},e,w)$ is a Sperner  $2$-connected $r^-$-graph satisfying (i)--(iii). Since
each pair in $e$ is thin, $\cH'$ is  Sperner. Properties (i)--(iii) are evident, so we need to check that $\cH'$ is $2$-connected.

By construction, $\cH'$ can be obtained from the $2$-connected $\cH_{i+1}$ by blowing up vertex $y^*$ into vertices 
$y_1,\ldots,y_{t-1}$ (each of a positive degree) and replacing edge $y^*y_t$ with $a$. In terms of the incidence graphs,
in the $2$-connected $I(\cH_{i+1})$, we split $y^*$ into $t-1$ vertices of degree at least $1$, delete vertex $y^*$ and add
vertex $a$ adjacent to $y_1,\ldots,y_t$. It is easy to check that the new graph is $2$-connected.

\medskip
{\bf Case 2:} $w= y^*$. By~(\ref{e4}), there is a unique $v_1\in a-y_t$ such that $e'=e-y^*+v_1\in \cH_{i-1}$. Since $e$
is unhappy in $\cH_i$, it has a pair $xy$ of codegree at most $|e|-2$. If $y^*\notin \{x,y\}$, then the codegree of $xy$ in
$\cH_{i-1}$ also is at most $|e|-2$. And if $y^*=y$, then the codegree of $y_1x$ in $\cH_{i-1}$  is at most $|e|-2$.
Thus $e'$ is unhappy in $\cH_{i-1}$. The rest is as in Case 1.\qed

\subsection{Stopping at $k-1$ vertices}
\begin{lemma}\label{c3k-1}
Suppose  $r\geq 3$ and $k\geq r$ are integers.
Let   $\cH$ be a Sperner  $2$-connected $r^-$-graph with $c(\cH)<k$ and at least $k$ vertices
that is not happy. Suppose $\cH= \cH_0, \ldots, \cH_i, \cH_{i+1}$ is a sequence of $r^-$-graphs obtained by iteratively applying Lemma~\ref{c3} following Rule~\eqref{e1} to $\cH$ until $\cH_{i+1}$ is happy. If {\rm (T5)} was never applied and $|V(\cH_{i+1})| = k-1$, then $|E(\cH_{i+1})| \leq {k-2 \choose \min\{r, \lfloor (k-2)/2 \rfloor\}} + 2$. 
\end{lemma}
%
%
\begin{proof}
Since (T1) does not change the number of vertices and $\cH_0$ has at least $k$ vertices,  one of (T2), (T3), or (T4) was applied.
Moreover, by Claims~\ref{cla1}--\ref{cla3}, one of (T2), (T3), or (T4) was applied to $\cH_i$ to obtain the happy $r^-$-graph $\cH_{i+1}$. For short, denote $\cH'=\cH_{i+1}$.

If $\cH'$ has a vertex of degree at most 3, then the number of edges in $\cH'$ is at most ${k-3 \choose \min\{r, \lfloor (k-3)/2 \rfloor\}} + {3 \choose \min\{r-1, 1\}}$, and we  are  done. Hence
\begin{equation}\label{e9}
\delta(\cH')\geq 3.
\end{equation}

In the following, for any $r^-$-graph $\cA$ and any vertex $v \in V(\cA)$, we use $\cA - v$ to denote the $r^-$-graph obtained by removing vertex $v$ and shrinking any edge $e$ that contains $v$ to the edge $e-v$, unless $|e| = 2$, in which case we simply delete $e$ in $\cA - v$. Note that $\cA-v$ need not be Sperner, even if $\cA$ is Sperner.

\medskip
{\bf Case 1:} (T4) was the last applied operation. Let $a=\{y_1,\ldots,y_t\}$ be the unhappy edge such that
$\cH'$ is obtained from $H_i$ by gluing $\{y_1,\ldots,y_{t-1}\}$ into a new vertex $y^*$. Since $\cH'$ is happy, $\cH_i - a$ is happy. The $r^-$-graph $F(\cH_i, a, y_t)$ satisfies (i)-(iii) and is Sperner by~\eqref{e4}. So if $F(\cH', a, y_t)$ is 2-connected, then we would have applied (T1) to $\cH_i$ instead 
of (T4), a contradiction to Rule~\eqref{e1}. Therefore
\begin{equation}\label{e5}
\parbox{15cm}{\em  the incidence graph $I(\cH_i-a)$ has a vertex $x_t$ separating $y_t$ from 
$\{y_1,\ldots,y_{t-1}\}$. }
\end{equation}
If $x_t$ corresponds to an edge $b$ in $\cH_i-a$, then some pair of its vertices is thin. So, since $\cH_i-a$ is happy,
$|b|=2$. Then instead of $x_t$, we can choose as a vertex $x'_t$ separating $y_t$ from  $\{y_1,\ldots,y_{t-1}\}$ the neighbor
of $x_t$ that is farther from $y_t$. 
Thus we may assume that  $x_t$ corresponds to a vertex in $\cH_i-a$. 

If $x_t \notin \{y_1, \ldots, y_{t-1}\}$, then $y_t$ and $y^*$ are also separated by $x_t$ in $\cH' - y^*y_t$. Since there are at least 2 components in $\cH' - y^*y_t - x_t$, the largest block of $\cH' - y^*y_t$ has at most $|V(\cH') - 1| = k-2$ vertices.

We have that \[|E(\cH')| = |E(\cH' - y^*y_t)| + 1 \leq {k-2 \choose \min\{r, \lfloor(k-2)/2 \rfloor\}} + 1 + 1 ={k-2 \choose \min\{r, \lfloor(k-2)/2 \rfloor\}} + 2.\]

If $x_t \in \{y_1,\ldots, y_{t-1}\}$, then let $\cC$ be a component of $(\cH_i - a) - x_t$ which does not contain $y_t$. Then $\cC$ contains a vertex $y \notin \{y_1, \ldots, y_{t-1}\}$, otherwise every edge of $\cC +x_t$ in $\cH_i$ would be a subset of the edge $a$, contradicting that $\cH_i$ is Sperner. Thus in $\cH' - y^*y_t$, $y$ and $y_t$ are in different blocks. Hence we again get $|E(\cH')| \leq {k-2 \choose \min\{r, \lfloor(k-2)/2 \rfloor\}} + 2.$

\medskip
{\bf Case 2:} $\cH_{i+1}=F(\cH_i,v,e_1,e_2)$ for some special vertex $v$. By~(\ref{e2}), if $|e_1|\geq 4$, then some pair 
in $e_1-v$ is thin, and hence $e_1-v$ is unhappy in $\cH_{i+1}$, a contradiction the happiness of $\cH_{i+1}$. Thus $|e_1|,|e_2|\leq 3$.
Since $\cH_i$ was unhappy, we may assume that $|e_1|=3$, say $e_1=\{v,v',v''\}$. By~(\ref{e2}), either $vv'$ or $vv''$ is a thin pair in
$\cH_i$. Suppose $vv''$ is thin. Consider $\cH''=F(\cH_i,e_1,v')$. Since $vv''$ is thin, $\cH''$ is Sperner. If $\cH''$ is $2$-connected, we get
a contradiction to Rule~(\ref{e1}). Thus the incidence graph $I(\cH'')$ has a cut vertex $x$ separating $v'$ from $\{v,v''\}$.
We claim that
\begin{equation}\label{e11}
\mbox{\em we can choose  $x$ corresponding to a vertex in $\cH''$ distinct from $v$. (We allow  $x=v''$.)}
\end{equation}
Indeed, if $v$ separates $v'$ from $v''$ in $I(\cH'')$, then vertex $e_1$ in the incidence graph $I(\cH_i)$ separates $v'$ from $v''$, a contradiction to the $2$-connectedness of $\cH_i$. 
If $x$ corresponds to an edge in $I(\cH_i)$, then again $x$ contains thin pairs. 
If $|x|\geq 3$. Then $x$ is unhappy. By the choice $\cH_{i+1}$, the only unhappy edge in $\cH''$ could be $e_2$. Recall that in this case,
$|e_2|=3$, say $x=e_2=\{v,w,w'\}$. But in this case, one of $v,w$ and $w'$ also separates $v'$ from $v''$, and we know that
it is not $v$. Recall that $vv''$ is a thin pair, and so $v''\notin \{w,w'\}$. Otherwise if $|x| = 2$, then both of its vertices are cut vertices. This proves~(\ref{e11}).

Recall that $|V(\cH'')|=|V(\cH_i)|=k$ and $e(\cH'')=e(\cH_i)\leq e(\cH_{i+1})+1$. Suppose first that each  component of $\cH'' - x$ has at least 3 vertices. Since $\cH'' - x$ has $k-1$ vertices and at least 2 connected components, $k \geq 7$, and the largest component of $\cH''-x$ has at most $k-4$ vertices. Therefore we obtain
$$e(\cH_{i+1})\leq e(\cH'')\leq {k-3\choose \min\{r, \lfloor (k-3)/2 \rfloor\}}+{4\choose 2}\leq {k-2\choose \min\{r, \lfloor (k-2)/2 \rfloor\}}+2.$$

Now suppose that some component $\cC$ of $\cH'' - x$ contains at most 2 vertices. By (\ref{e9}), $|\cC|=2$ and each of the two vertices in $\cC$ either has degree in $\cH''$ less than in $\cH_{i+1}$ or is $v$. But the only vertex having degree in $\cH''$ less than in $\cH_{i+1}$ is $v'$, and the vertices $v$ and $v'$ are
in distinct components of $\cH''-x$.

\medskip
{\bf Case 3:} $\cH_{i+1}=F(\cH_i,vu)$ for some special edge $vu$. Let $e_1$ be the unhappy edge incident to $v$ and 
$e_2$ be the unhappy edge incident to $u$. By~(\ref{e3}), all pairs in $e_1$ and $e_2$ are thin. So since $\cH_{i+1}$ is 
happy, $|e_1|=|e_2|=3$. Let $e_1=\{v,v',v''\}$ and $e_2=\{u,u',u''\}$, where possibly $v'=u'$.
As in Case 2, consider $\cH''=F(\cH_i,e_1,v')$. Since $vv''$ is thin, $\cH''$ is Sperner. If $\cH''$ is $2$-connected, we get
a contradiction to Rule~(\ref{e1}). Thus the incidence graph $I(\cH'')$ has a cut vertex $x$ separating $v'$ from $\{v,v''\}$.

Similarly to the proof of~(\ref{e11}), we derive
\begin{equation}\label{e12}
\parbox{14.5cm}{\em we can choose  $x$ corresponding to a vertex in $\cH''$ distinct from $v$ and $u$. (We allow  $x=v''$.)
}
\end{equation}
Furthermore, $x \notin \{u', u''\}$. Now $|V(\cH'')|=|V(\cH_i)|=k+1$ and $e(\cH'')=e(\cH_i)= e(\cH_{i+1})+1$.

Note that there cannot be any isolated vertices in $\cH'' - x$ since by (\ref{e9}), $\delta(\cH'')\geq 3$. Also, as in the previous case, there cannot be a component of $\cH'' - x$  with exactly 2 vertices. So we may assume that each component of $\cH'' - x$ has at least 3 vertices.

Let $\cC$ be the component of $\cH'' - x$ that contains $v$. Then $\cC$ must also contain $u$ and at least two of the vertices in $\{v'', u', u''\}$. Therefore $|\cC| \geq 4$. In particular, since $\cH'' - x$ contains exactly $k$ vertices and at least 2 connected components, $k \geq |\cC| + 3 \geq 7$.

As in Case 2, if the largest component of $\cH''-x$ has at most $k-4$ vertices (so $k \geq 8$ since $|\cC| \geq 4$), then
$$e(\cH_{i+1})\leq e(\cH'')\leq {k-3\choose \min\{r, \lfloor (k-3)/2 \rfloor\}}+{5\choose 2}\leq {k-2\choose \min\{r, \lfloor (k-2)/2 \rfloor\}}+2,$$ 
a contradiction. 

Now suppose a component $\cC'$ of  $\cH''-x$ has  $k-3$ or $k-2$ vertices. If $\cC'$ contains $v$, (i.e., $\cC' = \cC)$, then since $\cC$ contains $u$ as well, and $u$ and $v$ are incident to exactly 3 edges ($vu, e_1$, and $e_2$),
$$e(\cH''[\cC+x])\leq {|\cC'| - 2 + 1\choose  \min\{r, \lfloor (|\cC'| - 2 + 1)/2 \rfloor\}}+3.$$

For $|\cC'| = k-3$ we get  $$e(\cH'')\leq {k-4\choose  \min\{r, \lfloor (k-3)/2 \rfloor\}}+3+{4\choose 2}\leq {k-2\choose  \min\{r, \lfloor (k-2)/2 \rfloor\}}+2,$$ and for $|\cC'| = k-2$ we get $$e(\cH'')\leq {k-3\choose  \min\{r, \lfloor (k-3)/2 \rfloor\}}+3+{3\choose 2}\leq {k-2\choose  \min\{r, \lfloor (k-2)/2 \rfloor\}}+2.$$ 
So $\cC' \neq \cC$. But since $|\cC| \geq 4$, we have $|V(\cH'')| \geq |\cC'| + |\cC| + 1 \geq 4 + (k-3) + 1 = k+2$, a contradiction.
\end{proof}
\section{Proof of Theorem~\ref{main2conn}}

\begin{proof}Apply Lemma~\ref{c3} repeatedly to $\cH$ following Rule~\eqref{e1} to obtain an $r^-$-hypergraph $\cH'$ that is happy. By Lemma~\ref{c2'}, $\partial_2 \cH'$ has no cycle of length $k$ or longer.

Let $n_S$ and $m_S$ be the number of vertices and $r^-$-edges respectively that were deleted going from $\cH$ to $\cH'$ by applying operations (T1)-(T4), and let $n_B$ and $m_B$ be the number of vertices and $r^-$-edges respectively that were deleted from applying operation (T5). So $n = |V(\cH')| + n_S + n_B$ and $|E(\cH)| \leq N_{\rm{Sp}}(\partial_2\cH',r) + m_S + m_B$. If $|V(\cH')| \geq k$, then by Theorem~\ref{kopyblock} (applied to $\partial_2 \cH'$) and Lemma~\ref{c3}, we have

\begin{equation*}
|E(\cH')| \leq N_{\rm{Sp}}(\partial_2\cH',r) + m_S + m_S 
\end{equation*}
\begin{equation}\label{smallbig}
\leq \max \{f(|V(\cH')|,k,r,2), f(|V(\cH')|,k,r,t)\} + n_S + {t \choose \min\{r-1, \lfloor t/2 \rfloor\}} n_B
\end{equation}
First suppose that $n_B=0$, i.e., (T5) was never applied. Examining the coefficient of $n_S$ we see $1 \leq \min\{2, {t \choose \min\{r-1, \lfloor t/2 \rfloor\} } \}$.  So in the case $|V(\cH')|\geq k$, from~\eqref{smallbig}, we get $|E(\cH')| \leq \max\{f(n, r, k, 2), f(n, r, k, t)\}$, as desired. Otherwise, if $|V(\cH')| \leq k-1$, then either 
$$|E(\cH')| \leq {k-2 \choose \min\{r, \lfloor (k-2)/2 \rfloor\}} + 2 = f(k-1, k, r, 2)$$ by Lemma~\ref{c3k-1}, or 
$|V(\cH')| \leq k-2$ and $$|E(\cH')| \leq {|V(\cH')| \choose \min\{r, \lfloor |V(\cH')|/2 \rfloor\}} \leq f(|V(\cH')|, k, r, 2).$$ 
Either way we obtain $|E(\cH)| \leq f(n,k,r,2)$. 

So we may assume that at least one application of (T5) was required to obtain $\cH'$. 

Denote $H' := \partial_2 \cH'$ and let $Q$ be the $t$-core of $H'$ (that is, the resulting graph from applying $t$-disintegration to $H'$).
If $H'$ is $t$-disintegrable, i.e., $Q$ is empty, then $N_{\rm{Sp}}(H',r) \leq f(|V(H')|, k,r, t)$ and so by~\eqref{smallbig}, we get $|E(\cH)| \leq f(n, k, r, t)$.  So we may assume that $Q$ is non-empty. In particular, since $\delta(Q) \geq t+1$, $|V(Q)| \geq t+2$.

 \begin{claim}
The graph $Q$ is 1-hamiltonian. 
 \end{claim}
 \begin{proof}

First note that $|V(Q)| \leq k-1$: the case for $|V(H')| \leq k-1$ is trivial, and if $|V(H')| \geq k$, then by applying Kopylov's Theorem (Theorem~\ref{le:Kopy}), we obtain $|V(Q)| \leq k-2$. 

%

Next, we claim that $Q$ is 3-connected. If not, then there exists a cut set $\{x,y\}\subset V(Q)$ and at least two components in $H' - \{x,y\}$. Since $\delta(Q) \geq t+1$, for each of these components $C$, $|C \cup \{x,y\}| \geq t+2$. Hence $|V(Q)| \geq 2(t+2) - 2 \geq k$, a contradiction to $|V(Q)| \leq k-1$.

Therefore $Q$ is 3-connected. By Enomoto's Theorem (Theorem~\ref{eno}), $Q$ is $s$-path connected where $s = \min\{|V(Q)|, 2(t+1)\} = |V(Q)|$. I.e., $Q$ is 1-hamiltonian.
%
\end{proof}

Let $q := |V(Q)|$. 
Let $\cB$ be a special (in particular, happy) block that was removed in some application of (T5), and set $B = \partial_2 \cB$.  Let $x_B$ and $a_B$ be the vertex-edge cut pair corresponding to $\cB$, where some vertex $y_B \in V(\cB) \setminus V(\cH')$ is contained in $a_B$.

\begin{claim}\label{k-s+1}
Suppose $H'$ is $s$-path connected. There does not exist a $(x_B, y_B)$-path of length at least $k-s + 1$ in $B$.
\end{claim}
\begin{proof}
Since $\cH$ is 2-connected, its incidence bigraph contains two shortest disjoint paths $P_1$, $P_2$ from $\{x_B, a_B\}$ to $V(\cH')$ (where possibly $|V(P_1)$ or $V(P_2) = 1)$. Note that these paths are internally disjoint from $V(\cH') \cup V(\cB)$. In $\cH$, $P_1$ and $P_2$ yield Berge paths $\mathcal P_1$ and $ a \cup \mathcal P_2$ from $x_B$ to $V(\cH')$ and $y_B$ to $V(\cH')$ respectively. Say $P_i$ has endpoint $v_i \in V(\cH')$.

Now suppose there exists a path of length at least $k-s+1$ from $x_B$ to $y_B$. This yields a Berge path $\mathcal P_3$ from $x_B$ to $y_B$ with at least $k-s+1$ base vertices such that all edges of $\mathcal P_3$ are contained in $V(\cB)$. Similarly, we find a Berge path $\mathcal P_4$ from $v_1$ to $v_2$ with at least $s$ base vertices such that all edges of $\mathcal P_4$ are contained in $V(\cH')$.

Then $\mathcal P_1 \cup \mathcal P_3 \cup a \cup \mathcal P_2 \cup \mathcal P_4$ is a Berge cycle of length at least $(k-s+1) + s - 1 = k$, a contradiction.

\end{proof}

\begin{claim}\label{ssub}
If $H'$ contains a subgraph $S$ that is $s$-path connected, then $H'$ is also $s$-path connected.
\end{claim}

\begin{proof}
Let $\{x,y\} \subset V(H')$. We will show that there exists an $(x,y)$-path in $H'$ with at least $s$ vertices. Let $P_x, P_y$ be two disjoint shortest paths from $\{x,y\}$ to $V(S)$, say with endpoints $v_x$ and $v_y$ respectively (where possibly one or both paths are singletons). Such paths exist because $H'$ is 2-connected. Let $P_S$ be a $(v_x, v_y)$-path in $S$ of length at least $S$. Then $P_x \cup P_S \cup P_y$ has length at least $s$.  
\end{proof}

Therefore the previous claim shows that $H'$ is $q$-path connected. Applying Claim~\ref{k-s+1} and Theorem~\ref{kpaththm}, we get 
\begin{equation}
e(\cB) \leq N_{\rm{Sp}}(B, r) \leq \frac{|V(B)| - 2}{k-q-2}{k-q \choose \min\{r, \lfloor (k-q)/2 \rfloor\}}. 
\end{equation}
 Summing up over all blocks deleted via big cuts, we obtain
\begin{equation}\label{Bineq}
m_B \leq n_B \left(\frac{1}{k-q-2}{k-q \choose \min\{r, \lfloor (k-q)/2 \rfloor\}}\right)
\end{equation}

\begin{claim}\label{sineq} For each  integer $s \geq 3$, $\frac{1}{s-2}{s \choose \min \{r, \lfloor s/2 \rfloor\}} \leq {s \choose \min\{r-1, \lfloor s/2 \rfloor\}}$. 
\end{claim}

\begin{proof}
The case for $\min\{r, \lfloor s/2 \rfloor\} = \lfloor s/2 \rfloor$ is trivial. So we may assume $s \geq 2r+2$. We have $\frac{1}{s-2}{s \choose r} = \frac{1}{s-2}\frac{s-r+1}{r} {s \choose r-1} \leq {s \choose r-1}$. 
\end{proof}

So first suppose that $|V(\cH')| \geq k$. By Kopylov's theorem, $t+2 \leq q \leq k-2$, and $V(H') - V(Q)$ can be removed via $(k-s)$-disintegration. Therefore
\[e(\cH') \leq {q \choose \min \{r, \lfloor q/2 \rfloor\}} + (|V(\cH')|-q){k-q \choose \min \{r-1, \lfloor (k-q)/2 \rfloor\}},\]

and hence by~\eqref{Bineq} and the previous claim, \[e(\cH) = e(\cH') + m_B + m_S \leq 
\]\[\leq 
{q \choose \min \{r, \lfloor q/2 \rfloor\}} + (|V(\cH')|-q){k-q \choose \min \{r-1, \lfloor (k-q)/2 \rfloor\}} +
n_B \left(\frac{1}{k-q-2}{k-q \choose \min\{r, \lfloor (k-q)/2 \rfloor\}}\right) + n_S\]
\[\leq {q \choose \min \{r, \lfloor q/2 \rfloor\}} + (n-q){k-q \choose \min \{r-1, \lfloor (k-q)/2 \rfloor\}} \leq \max\{f(n,k,r,t), f(n,k,r,2)\},\]
where the last inequality follows from the convexity of the function $f$. So from now on we may assume $|V(H')| \leq k-1$.
\begin{claim}Let $S$ be a 1-hamiltonian subgraph of $H'$ with $s:=|V(S)|$ and $t+2 \leq s \leq k-2$. Let $S'$ be the result of $(k-s)$-disintegration applied to $H'$. Then $S'$ is also 1-hamiltonian. 
\end{claim}
\begin{proof}
We will show a stronger statement: $S'$ is $(k-|V(S')|)$-hamiltonian. Suppose not. Set $s':= |V(S')|$. Applying Theorem~\ref{PPS} with $d = k-s$ (so $d \leq 2t+2 - (t+2) = t$) and $\ell = k-s'$, we get 
\[N_{\rm{Sp}}(S', r) \leq \max\{ h_{\rm{Sp}}(s', k-s',r,k-s), h_{\rm{Sp}}(s',  k-s',r,\lfloor s'/2 \rfloor,)\}.\]
If $ h_{\rm{Sp}}(q', k-s',r,k-s) \geq  h_{\rm{Sp}}(s', k-s', r,\lfloor s'/2 \rfloor)$, then 

\begin{eqnarray*}
N_{\rm{Sp}}(S',r) &\leq & h_{\rm{Sp}}(s', k-s', r,k-s) \\
&=& {s \choose \min\{r, \lfloor s/2 \rfloor\}} + (s' - s) {k-s \choose \min\{r-1, \lfloor (k-s)/2 \rfloor\}}\\
& = & f(s', k, r, k-s),
\end{eqnarray*}

Recall that since $S$ is 1-hamiltonian, $H'$ is $s$-path connected. Hence for each $\cB$ deleted in an application of (T5), $\partial_2 \cB$ is not $(k-s+1)$-path connected.

It follows that \[e(\cH) \leq N_{\rm{Sp}}(H', r) + m_B + m_S \]\[\leq f(s', k, r, k-s) + (|V(H')|-s' + n_B){k-s \choose \min\{r-1, \lfloor (k-s)/2 \rfloor\}} + n_S \leq f(n, k, r, k-s).\]
So by the convexity of the function $f$, we are done.

Next suppose $h_{\rm{Sp}}(s', k-s',r, k-s) \leq h_{\rm{Sp}}(s',k-s', r, \lfloor s'/2 \rfloor)$. For simplicity, let $a := \lfloor s'/2 \rfloor$. We have that $2 \leq a \leq \lfloor (k-1)/2 \rfloor = t$. 

\begin{eqnarray*}
N_{\rm{Sp}}(S', r) &\leq & h_{\rm{Sp}}(s', k-s',r,a) \\ 
& = & {s'-(a-k+s') \choose \min\{r, \lfloor (s'-(a-k+s'))/2 \rfloor \}} + (a - k + s') {a \choose \min\{r-1, \lfloor a/2 \rfloor\}}\\
& = & {k-a \choose \min\{r,\lfloor (k-a)/2 \rfloor\}} + (s' - (k-a)){a \choose \min\{r-1, \lfloor a/2 \rfloor\}}\\ 
& \leq & f(s',k,r,a) \leq  f(s', k, r, t). 
\end{eqnarray*}

Therefore \[e(\cH) \leq f(s', k, r, t) + (|V(H')|-s' + n_B){k-s \choose \min\{r-1, \lfloor (k-s)/2 \rfloor\}} + n_S \leq f(n, k, r, t).\]
\end{proof}

Starting from the 1-hamiltonian subgraph $Q$ of $H'$, we obtain a sequence of graphs $Q = Q_0 \subset Q_1 \subset \ldots \subset Q_q$ such that $Q_i$ is the resulting 1-hamiltonian subgraph obtained from $(k- |V(Q_{i-1})|)$-disintegration applied to $H'$. The sequence ends when either the graph $Q_{q+1}$ resulting from the $(k-|V(Q_q)|)$-disintegration of $H'$ is exactly $Q_{q}$, or $|V(Q_{q})| = k-1$. 
In the former case, we have that $|V(Q_{q+1})| = |V(Q_q)| =: q'$. Then

\[e(\cH) \leq N_{\rm{Sp}}(H', r) + m_B + m_S \]\[\leq f(q', k, r, k-q') + (|V(H')|-q' + n_B){k-q' \choose \min\{r-1, \lfloor (k-q')/2 \rfloor\}} + n_S \leq f(n, k, r, k-q').\]


Finally suppose that $|V(Q_{q})| =k- 1$. Then $H'$ is $(k-1)$-path connected. Because $\cH'$ is 2-connected, we can complete a Berge path in $\cH'$ with at least $k-1$ vertices to a Berge cycle of length at least $k$.
This proves the theorem.

\end{proof}
\section{Proof of Theorem~\ref{main_paths} for paths}

%
\begin{proof}
 Let $\cH$ be a counterexample of Theorem~\ref{main_paths} with minimum $\sum_{e \in E(\cH)} |e|$ on at least $k+1$ vertices. If $\cH$ contains a Berge cycle of length $k+1$ or longer, then removing any edge from this Berge cycle yields a Berge path with at least $k+1$ base vertices, a contradiction. If $\cH$ contains a Berge cycle of length exactly $k$, then we use the following Lemma which contradicts that $n:= |V(\cH)| \geq k+1$. 
\begin{lemma}[Gy\H{o}ri, Katona, and Lemons~\cite{lemons}]
Let $\mathcal H$ be a connected hypergraph with no Berge path of length $k$. If there is a Berge cycle of length $k$ on the vertices $v_1, \ldots, v_{k}$ then these vertices constitute a component of $\mathcal H$. 
\end{lemma}

Therefore $\cH$ contains no Berge cycle of length $k$ or longer. If $\cH$ is 2-connected, then by Theorem~\ref{main2conn}, $e(\cH) \leq \max\{f(n,k,r,2), f(n,k,r,\lfloor(k-1)/2 \rfloor)\}$, and we are done. 

Now suppose $\cH$ is not 2-connected. Then the incidence bigraph $I_{\cH}$ of $\cH$ contains a set of cut vertices. 
If a cut vertex $x$ of $I_{\cH}$ corresponds to an edge in $\cH$, then we say $x$ is a cut edge of $\cH$. Otherwise, we say $x$ is a cut vertex of $\cH$. 

Suppose $\cH$ has an cut-edge $e$. We claim that for each  component $\cC$ of $\cH \setminus e$, 

\begin{equation}\label{cutedge}
|V(\cC) \cap e| \leq 1.
\end{equation}

Indeed suppose that some component $\cC$ of $\cH \setminus e$ contains at least 2 vertices in $e$. Let $\cH'$ be the $r^-$-graph obtained by shrinking $e$ to remove all but one vertex in $\cC$ from $e$. Then $\cH'$ is still connected and Sperner (since $e$ is a cut edge of $\cH$). Furthermore, after this operation, the length of a longest path cannot increase. This contradicts the choice of $\cH$.

Now suppose $\cH$ contains a cut edge $e$. By ~\eqref{cutedge}, $e$ intersects every component of $\cH \setminus e$ in at most one vertex. Let $\cH'$ be the $r^-$-graph obtained by contracting two vertices of $e$ into a single vertex (and then deleting $e$ if it now contains only one vertex). The new $r^-$-graph $\cH'$ is Sperner, contains no Berge $P_k$, and is connected. If $|V(\cH')| \geq k+2$, we obtain that $\cH'$ contradicts the choice of $\cH$ (note that $e(\cH') \geq e(\cH) - 1 \geq \max\{f(n,k,r,1), f(n, k, r, \lfloor(k-1)/2 \rfloor)\} - 1 \geq \max\{f(n-1,k,r,1), f(n-1, k, r, \lfloor(k-1)/2 \rfloor)\})$.

Iterating this process,  we may assume that $\cH$ contains no cut edges unless $n = k+1$.

{\bf Case 1:} $\cH$ does not have a cut edge. 
 
Any block $\cB$ of $\cH$ is a subhypergraph of $\cH$. In particular, $\cB$ is a Sperner 2-connected $r^-$-graph. Let $\cB_1, \ldots, \cB_p$ be the blocks of $\cH$. For each $i$, let $s_i$ be the length of a longest Berge cycle in $\cB_i$. Without loss of generality, we may assume $s_1 \geq \ldots \geq s_p$.

\begin{claim}For all $i \geq 2$, $s_1 + s_i \leq k+1$. 
\end{claim}
In particular, $s_i \leq (k+1)/2$ for all $i \geq 2$. 

\begin{proof}Suppose $s_1 + s_i \geq k+2$. 
Let $C_1$ be a Berge cycle of $\cB_1$ of length $s_1$ and let $C_i$ be a Berge cycle of $\cB_i$ of length $s_i$. Let $P$ be a shortest Berge path from $V(\cB_1)$ to $V(\cB_i)$. Note that $P$ contains at most one edge from each Berge cycle. Then removing an edge from each Berge cycle, we obtain together with $P$ a Berge path whose base vertices cover $V(C_1) \cup V(C_i)$. Since $|V(C_1) \cap V(C_i)| \leq 1$, this path has at least $s_1 + s_i - 1 \geq k+1$ base vertices. 
\end{proof}

For each block $\cB_i$, let $n_i := |V(\cB_i)|$. If $n_i = s_i$, then \[e(\cB_i) \leq {s_i \choose \min\{r, \lfloor s_i/2 \rfloor\}} \leq (n_i-1) {s_i-1 \choose \min\{r-1, \lfloor (s_i-1)/2 \rfloor\}}. \]
  If $n_i \geq s_i + 1$, then we apply Theorem~\ref{main2conn} to $\cB_i$ with cycle length $s_i + 1$. We obtain \[e(\cB_i) \leq \max\{f(n_i, s_i +1, r, 2), f(n_i, s_i+1,\lfloor s_i/2 \rfloor\}.\]
Furthermore, 

\begin{eqnarray*}
f(n_i, s_i + 1, r, 2) &=& {s_i - 1 \choose \min\{r, \lfloor (s_i-1)/2 \rfloor\}} + 2(n_i - s_i + 1)\\
& \leq &  (s_i-1){s_i - 2 \choose \min\{r-1, \lfloor (s_i - 2)/2 \rfloor\}} + (n_i - s_i) {s_i - 2 \choose \min\{r-1, \lfloor (s_i - 2)/2 \rfloor\}} \\
&=& (n_i - 1) {s_i - 2 \choose \min\{r-1, \lfloor (s_i - 2)/2 \rfloor\}}.
\end{eqnarray*}
And $f(n_i, s_i+1, r, \lfloor s_i/2 \rfloor) \leq (n_i - 1) {s_i - 1 \choose \min\{r-1, \lfloor (s_i - 1)/2 \rfloor\}}$.

In all cases we get 
\begin{equation}\label{bi}
e(\cB_i) \leq (n_i - 1) {s_i - 1 \choose \min\{r-1, \lfloor (s_i - 1)/2 \rfloor\}}.
\end{equation}
%

For $\cB_1$, if $n_1 = s_1$ then $e(\cB_1) \leq {s_1 \choose \min\{r, \lfloor s_1/2 \rfloor\}}$ and so by~\eqref{bi},

\begin{equation}\label{s1si}e(\cH) \leq  {s_1 \choose \min\{r, \lfloor s_1/2 \rfloor\}} + \sum_{i=2}^p (n_i-1) {s_i - 1 \choose \min\{r-1, \lfloor (s_i - 2)/2 \rfloor\}}.
\end{equation}

If $s_1 \geq \lceil (k+1)/2 \rceil$, then from~\eqref{s1si} we obtain

\[e(\cH) \leq  {s_1 \choose \min\{r, \lfloor s_1/2 \rfloor\}} + \sum_{i=2}^p (n_i-1) {k-s_1 \choose \min\{r-1, \lfloor (k-s_1)/2 \rfloor\}} \leq f(n, k, r, k-s_1)\]
\[\leq \max\{f(n, k, r, 1), f(n, k, r, \lfloor (k-1)/2 \rfloor\}).\]

Otherwise, \[e(\cH) \leq  {s_1 \choose \min\{r, \lfloor s_1/2 \rfloor\}} + \sum_{i=2}^p (n_i-1) {s_1-1 \choose \min\{r-1, \lfloor (s_1-1)/2 \rfloor\}} \leq f(n,k,r, \lfloor (k-1)/2 \rfloor).\]

If $n_1 \geq s_1 + 1$, then we get

\[e(\cB_1) \leq \max\{f(n_1, s_1+1, r, 2), f(n_1, s_1+1, r, \lfloor s_1/2 \rfloor\}).\] 

If $f(n_1, s_1 + 1, r, \lfloor s_1/2 \rfloor) \geq f(n_1, s_1+1, r, 2)$, then together with~\eqref{bi}, we get

\[e(\cH) \leq f(n_1, s_1+1, r, \lfloor s_1/2 \rfloor) + \sum_{i=2}^p  (n_i-1){\lfloor\frac{k-1}{2}\rfloor \choose \min\{r-1, \lfloor\frac{k-1}{4}\rfloor\}} \leq f(n, k, r, \lfloor (k-1)/2 \rfloor).\]

If $f(n_1, s_1 + 1, r, \lfloor s_1/2 \rfloor) < f(n_1, s_1+1, r, 2)$,   then \[f(n_1, s_1+1, r, 2) = {s_1 - 1 \choose \min\{r, \lfloor (s_1-1)/2 \rfloor\}} + 2 (n_1 - s_1 + 1) \leq {s_1 \choose \min\{r, \lfloor s_1/2 \rfloor\}} + 2(n_1 - s_1).\]

Thus we obtain

\[e(\cH) \leq {s_1 \choose \min\{r, \lfloor s_1/2 \rfloor\}} + 2(n_1 - s_1)+ \sum_{i=2}^p (n_i-1) {s_i - 1 \choose \min\{r-1, \lfloor (s_i - 1)/2 \rfloor\}},\] and we are done as in the the case for~\eqref{s1si}.

{\bf Case 2}: $n = k+1$ and $\cH$ contains a cut edge.

Let $e$ be a cut edge of $\cH$. By~\eqref{cutedge}, each component $\cC$ of $\cH \setminus e$ contains only at most one vertex of $e$. If $|e| \geq 3$, then $e(\cH \setminus e) \leq {k+1 - 2 \choose \min\{r, \lfloor (k+1 - 2)/2 \rfloor\}}$. Hence
$e(\cH) \leq {k-1 \choose \min\{r, \lfloor (k-1)/2 \rfloor\}} + 1 < f(n,k,r,1)$.

So we may assume $|e| = 2$. Suppose first that  $\cH \setminus e$ contains a component $\cC$ with $2 \leq |V(\cC)|\leq k-1$. 

Then \[e(\cH) \leq 1 + {|V(\cC)| \choose \min\{r, \lfloor |V(\cC)|/2 \rfloor\}} + {(k+1) - |V(\cC)| \choose \min\{r, \lfloor ((k+1) - |V(\cC)|)/2 \rfloor\}} \leq 1 + {k-1 \choose \min\{r, \lfloor (k-1)/2 \rfloor\}} + 1\] \[ = f(n,r,k, 1).\]

Thus $\cH \setminus e$ must consist of one component of size $k$ and one of size 1. The same also holds for every other cut edge $e'$ of $\cH$. This together with~\eqref{cutedge} implies that if $\cH$ has two cut edges $e, e'$, then $e'$ is a cut edge of $\cH \setminus e$, and vice versa. Therefore $e(\cH) \leq {k-1 \choose \min\{r, \lfloor (k-1)/2 \rfloor\}} + 2 = f(n,k,r,1)$. 

So we may assume that $e$ is the only cut edge of $\cH$. Let $\cC$ be the component of $\cH$ of size $k$. This component cannot contain a Berge cycle of length $k$, otherwise with $e$ we would obtain Berge path with of length $k$.

If $\cC$ is 2-connected, then by Theorem~\ref{main2conn},
\[e(\cH) = e(\cC) + 1 \leq \max\{f(k,k,r,2), f(k, k, r, \lfloor (k-1)/2 \rfloor)\} < f(n, k, r, 1).\]

Otherwise $\cC$ has a cut vertex $v$ and a block $\cB$ with $2\leq|V(\cB)|\leq k-1$. Therefore
\[e(\cC) \leq {|V(\cB)| \choose \min\{r, \lfloor |V(\cB)|/2 \rfloor\}} + { k-|V(\cB)|+1 \choose \min\{r, \lfloor (k-|V(\cB)|+1)/2 \rfloor\}} \leq {k-1 \choose \min\{r, \lfloor (k-1)/2 \rfloor\}} + 1,\] so we get $e(\cH) = e(\cC) + 1 \leq f(n,k,r,1)$. 
This proves the theorem.
\end{proof}
\section{Concluding remarks}

\begin{enumerate}
\item As it is mentioned in Theorem~\ref{c23}, if $ k \geq 4r$ and $n$ is asymptotically larger than $\frac{2^{r-1}}{r}k$, then our
bound is also exact for $r$-graphs: a
 sharpness example is $\cH_{n,k,r,\lfloor (k-1)/2 \rfloor}$. We think that for smaller $n$, our bound for $r$-graphs is not exact.
 It would be interesting and challenging to find exact bounds for the number of edges in $n$-vertex 
 $2$-connected $r$-graphs with no cycles of length $k$ or longer for $k>r$ and  $k\leq n<\frac{2^{r-1}}{r}k$.
 
 \item When $r$ is large, $k\geq 4r$ and $n$ is polynomial in $k$, then $\cH_{n,k,r,2}$
 has not much more than
  ${k-2\choose r}$ edges. Also $\cH_{n,k,r,2}$ is not uniform whenever $r \geq 4$. 
 The following  construction of $2$-connected $r$-uniform hypergraphs  also has  more than ${k-2\choose r}$ edges in this case,
 although fewer  edges than $\cH_{n,k,r,2}$ has (and it works only for $n$ such that $n-k+2$ is divisible by $r-1$). 
 
\begin{const}\label{cons3}
Fix $k\geq 4r\geq 12$, $s\geq 1$, $n = k-2+s(r-1)$. Define the $n$-vertex $r$-graph $F_{n,k,r,s}$ as follows.
 The vertex set of $F_{n,k,r,s}$ is partitioned into $s+1$ sets $A_1,\ldots, A_s,C$ such that $|C| = k-2$ and $|A_i| = r-1$ 
 for all $ i\in [s]$. We fix two special vertices $c_1,c_2\in C$.
 The edge set of $F_{n,k,r,s}$ consists of all edges contained in $C$ and of the $2(r-1)$ edges of the form
  $A_i \cup \{c_j\}$ for $i\in [s]$ and $j\in [2]$.
 \end{const}
 We do not currently know of any {\em uniform} hypergraphs with more edges and no Berge cycles of length $k$ or longer.
 
 \item  Note that  here we use $r^-$-graphs to prove a bound for $r$-graphs when $k>r$ and in~\cite{KL} we used
 $r^+$-graphs (i.e. hypergraphs with the lower rank at least $r$) in the case $k<r$.
 
 \end{enumerate}

\newpage
\section*{Appendix about convexity}
\begin{claim}\label{convex}For fixed positive integers $n$, $k$, and $r$,
the function \[f(n,k,r,a) = {k-a \choose \min\{r, \lfloor\frac{k-a}{2}\rfloor\}} + (n-k+a){a \choose \min\{r-1, \lfloor a/2 \rfloor\}}\] is convex over integers 
$\max\{ 0, k-n\} \leq a \leq k$. 
 \end{claim}

 In particular, if we consider $f(n,k,r,a)$ over a domain of integers, say $\{c, \ldots, d\}$ where $c, d\in \mathbb Z$, 
$\max\{ 0, k-n\} \leq c\leq d \leq k$
then $f(n,r,k,a)$ attains its maximum at either $a = c$ or $a = d$.

\begin{proof}
Since we only consider integer values of $a$, we may view $f(n,k,r,a)$ as a sequence of numbers. 

We say a sequence of real numbers $(f_i)_{i=u}^v$ is {\em convex} if $f_{i-1} + f_{i+1} \geq 2f_i$ for all $u < i < v$.

\begin{fact}\label{fusion}Let $u < v < w$ be integers. Suppose $(f_i)_{i=u}^{v+1}$ and $(g_i)_{i=v}^w$ are convex sequences of real numbers such that $f_v = g_v$ and $f_{v+1} = g_{v+1}$. Then the sequence $(h_i)_{i=u}^w$ where \[h_i:= \left\{ \begin{array}{ll} 
f_i, & u \leq i \leq v+1 \\ 
g_i, & v \leq i \leq w
\end{array}\right.\] is convex. 

\end{fact}
Indeed for any $u < i < w$, either $(h_{i-1}, h_i, h_{i+1}) = (f_{i-1}, f_i, f_{i+1})$ or $(h_{i-1}, h_i, h_{i+1}) = (g_{i-1}, g_i, g_{i+1})$.

The following two facts are easy to check. 
\begin{fact}The sequence $(x_i)_{i=0}^\infty$ where $x_i:={i \choose \lfloor i/2 \rfloor}$ is convex. 
\end{fact}

\begin{fact}For any fixed positive integer $r$, the sequence $(y_i)_{i=0}^\infty$ where $y_i = {i \choose r}$ is convex. 
\end{fact}

By Facts 9.3--9.5, function $g_1(a) := {k-a \choose \min\{\lfloor\frac{k-a}{2}\rfloor, r\}}$ is convex for integers 
$0\leq a \leq k$,
and $g_2(a):={a \choose \min\{r-1, \lfloor a/2 \rfloor\}}$ is convex for integers $a \geq 0$. Here we use that $g_1(a) = {k-a \choose r}$ when $a \leq k - 2r$ and $g_1(a) = {k-a \choose \lfloor (k-a)/2 \rfloor}$ when $a \geq k-2r-1$. One can show similar cut-offs for $g_2(a)$.

Note that $g_2(a)$ is non-decreasing. We also show that $(n-k+a) \cdot g_2(a)$ is convex for integers 
$a\geq \max\{ 0, k-n\}$:
\begin{eqnarray*}
& & (n-k+(a-1)) \cdot g_2(a-1) + (n-k+(a+1)) \cdot g_2(a+1)  \\
&= &(n-k + a) \cdot (g_2(a-1) + g_2(a+1)) - g_2(a-1) + g_2(a+1)\\
&\geq & (n-k+a)\cdot(2g_2(a)) + 0\\
&=& 2 (n-k+a) \cdot g_2(a).
\end{eqnarray*} 
Since the sum of two convex sequences is also convex, function $g_1(a) + (n-k+a) \cdot g_2(a) = f(n,k,r,a)$ is also convex for integers 
$\max\{ 0, k-n\} \leq a \leq k$. 
This proves the claim.
\end{proof}
\end{document}